\documentclass[10pt]{article}

\usepackage{boites}
\usepackage{stmaryrd}
\usepackage{amsmath,amssymb,amsthm,mathtools}

\newtheorem{condition}{Condition}
\newtheorem{theorem}{Theorem}[section]
\newtheorem{lemma}{Lemma}[section]

\newtheorem{proposition}{Proposition}[section]

\newcommand\reallywidehat[1]{\arraycolsep=0pt\relax%
	\begin{array}{c}
		\stretchto{
			\scaleto{
				\scalerel*[\widthof{\ensuremath{#1}}]{\kern-.5pt\bigwedge\kern-.5pt}
				{\rule[-\textheight/2]{1ex}{\textheight}} 
			}{\textheight} %
		}{0.5ex}\\           
		#1\\                 
		\rule{-1ex}{0ex}
	\end{array}
}
\newcommand{\norm}[1]{\left\lVert#1\right\rVert}
\def\R{\mathbb{R}}
\def\C{\mathbb{C}}

\def\R{\mathbb{R}}

\def\P{\mathbb{P}}

\def\L{\mathbb{L}}
\def\1{\mathbbm{1}}
\def\l{\mathcal{L}}
\def\c{\mathcal{C}}

\def\d{\mathcal{D}}
\def\v{\mathcal{V}}
\def\f{\mathcal{F}}
\def\u{\mathcal{U}}
\def\r{\mathcal{R}}

\def\v{\mathcal{V}}
\def\a{\mathcal{A}}

\newcommand{\bs}{\begin{eqnarray*}}
	\newcommand{\es}{\end{eqnarray*}}

\newcommand{\bx}{\mathbf{x}}

\newcommand{\bpi}{\mbox{\boldmath $\pi$}}
\newcommand{\bmu}{\mbox{\boldmath $\mu$}}

\newcommand*{\defeq}{\mathrel{\vcenter{\baselineskip0.5ex \lineskiplimit0pt
			\hbox{\scriptsize.}\hbox{\scriptsize.}}}%
	=}

\def\bb{\mathbf{b}}
\def\R{\mathbb{R}}

\def\bU{\mathbf{U}}
\def\H{\mathbb{H}}
\def\C{\mathbb{C}}

\def\R{\mathbb{R}}

\def\P{\mathbb{P}}

\def\L{\mathbb{L}}
\def\1{\mathbbm{1}}
\def\l{\mathcal{L}}
\def\c{\mathcal{C}}
\def\be{\mathbf{e}}

\def\d{\mathcal{D}}
\def\v{\mathcal{V}}

\def\u{\mathcal{U}}
\def\r{\mathcal{R}}

\def\v{\mathcal{V}}
\def\a{\mathcal{A}}

\def\f{\mathbf{f}}
\newcommand{\beq}{\begin{equation}}
\newcommand{\eeq}{\end{equation}}
\newcommand{\beqs}{\begin{equation*}}
\newcommand{\eeqs}{\end{equation*}}
\newcommand{\beqa}{\begin{eqnarray}}
\newcommand{\eeqa}{\end{eqnarray}}
\newcommand{\beqas}{\begin{eqnarray*}}
\newcommand{\eeqas}{\end{eqnarray*}}

\newcommand{\integer}{\hbox{N \kern -.em I}\ }  
\newcommand{\real}{\hbox{R \kern -.2em I}\ }  
\newcommand{\numbersetmm}{\,\hbox{\scriptsize M  \kern -.2em I}\,\ }
\newcommand{\numbersetm}{\,\hbox{M \kern -.2em I}\,\,\ }

\newcommand{\bo}{{\mathbf 0}}
\newcommand{\bv}{{\mathbf v}}
\newcommand{\bw}{{\mathbf w}}
\newcommand{\ep}{\epsilon}

\newcommand{\bLam}{\mbox{\boldmath $\Lambda$}}

\newcommand{\pr}{\partial}
\newcommand{\bA}{\mbox{\boldmath $A$}}

\newcommand{\bphi}{\mbox{\boldmath $\Phi$}}

\newcommand{\bI}{\mbox{\boldmath $I$}}
\newcommand{\bK}{\mbox{\boldmath $K$}}

\newcommand{\bq}{\mbox{\boldmath $q$}}

\newcommand{\btau}{\mbox{\boldmath $\tau$}}
\newcommand{\goto}{\rightarrow}
\newcommand{\B}{\mathbf}

\def\R{{\mathbb{R}}}

\def\C{{\mathbb{C}}}
\def\bA{\mathbf{A}}
\def\bB{\mathbf{B}}

\def\bU{\mathbf{U}}

\def\f{\mathbf{f}}

\def\bg{\mathbf{g}}
\def\bq{\mathbf{q}}
\def\bs{\mathbf{s}}

\def\by{\mathbf{y}}
\def\bn{\mathbf{n}}
\def\bx{\mathbf{x}}
\def\bu{\mathbf{u}}
\def\bv{\mathbf{v}}

\def\bo{\mathbf{0}}
\def\bI{\mathbf{I}}
\def\R{\mathbb{R}}

\usepackage{authblk}
\title{Integral Representation of Hydraulic Permeability} 
\author{Chuan Bi}
\author{M. Yvonne Ou}
\author{Shangyou Zhang}
\affil{Department of Mathematical Sciences, University of Delaware, Newark, DE 19716, USA}
\affil[ ]{\textit {bichuan@udel.edu,  mou@udel.edu, szhang@udel.edu}}
\renewcommand\Affilfont{\itshape}

\begin{document}

\maketitle 
Version of: \today

\begin{abstract}
In this paper, we show that the permeability of a porous {material} \cite{tartar1980incompressible} and that of a bubbly fluid 
\cite{Lipton_Avellaneda_1990} are limiting cases of the complexified version of the two-fluid models posed in  \cite{Lipton_Avellaneda_1990}. We assume the viscosity of the inclusion fluid is $z\mu_1$ and the viscosity of the hosting fluid is $\mu_1\in \mathbb{R}^+$, $z\in\mathbb{C}$. The proof  is carried out by the construction of solutions for large $|z|$ and small $|z|$ with an iteration process similar to the one used in  \cite{bruno1993stiffness,golden1983bounds} and the analytic continuation.  Moreover, we also show that for a fixed microstructure, the permeabilities of these three cases share the same integral representation formula (IRF) \eqref{IRF_Ks_prime} with different values of contrast parameter $s:=1/(z-1)$, as long as $s$ is outside the interval  $[-\frac{2E_2^2}{1+2E_2^2},-\frac{1}{1+2E_1^2}]$, where the positive constants $E_1$ and $E_2$ are the extension constants that depend only on the geometry of the periodic pore space of the material.
\end{abstract}


Classification: {35Q35, 35Q70}

\maketitle

\section{Introduction}
Darcy's law, which was first proposed by {H.} Darcy in 1856 \cite{darcy1856fontaines} based on experimental observation of water flowing through beds of sand, describes the relationship between the spontaneous flow discharge rate of steady state through a porous medium, the viscosity of the fluid, and the pressure drop over a distance. 
Later, theoretical/mathematical derivations of Darcy's law were presented in many works,  e.g. M. Poreh et. al \cite{poreh1965analytical}, S.P. Neuman\cite{neuman1977theoretical}, {E.} Sanchez-Palencia\cite{Sanchez-Palencia_1980}, J.L. Lions \cite{lions1981some}, J.B. Keller\cite{keller1980darcy} and J-L Auriault et al \cite{auriault1985dynamics-of-por}, just to name a few. 

In the setting of a periodic pore microstructure, as the period goes to zero, the convergence to the Darcy's law  of the Stokes system with no-slip boundary condition posed on {the} boundary of the pore space  was proved by  L. Tartar using  the energy method \cite{tartar1980incompressible}. G. Allaire implemented the two-scale convergence method introduced by {G.} Nguetseng\cite{nguetseng1989general} to derive the Darcy's law and show the convergence \cite{allaire1992homogenization,Allaire2010lecture}. Prior to the proof of Darcy's law in the {'80s}, {H.} Brinkman\cite{brinkman1949calculation} studied the viscous force exerted by a flowing flow on a dense swarm of particles by adding a diffusion term to the Darcy's law so as to take into account the transitional flow between boundaries. {H.} Brinkman's method was further studied in \cite{tam1969drag, saffman1971boundary, lundgren1972slow}. In the case of a porous {material} where the solid region is much smaller than the fluid part, {T.} Levy \cite{levy1983fluid} and {E.} Sanchez-Palencia\cite{Sanchez-Palencia_1980} proposed the same form of Darcy's law but with a different representation of {the} permeability tensor $\bK$. Later on, {G.} Allaire\cite{allaire1991continuity} showed the continuity of the transition between the two forms of Darcy's laws by considering various ratios between the size of the solid inclusion and the size of {the} separation. Moreover, instead of considering the porous {materials } as a periodic structured material, {A.} Beliaev \cite{beliaev1996darcy} considered the porous {materials} as a random and stochastically homogeneous material and deduced the same Darcy's law. {G.} Allaire \cite{allaire1989homogenization} generalized the homogenization to handle the more realistic micro-geometries of the porous medium where both the solid part and the fluid part are connected. Furthermore, in terms of the the fluid-solid interface conditions, a slip boundary condition is considered by {G.} Allaire in \cite{allaire1989homogenization,cioranescu1996homogenization}. 
In the case of the fluid flow through a porous medium subject to a time-harmonic pressure gradient, the permeability depends on the frequency and is referred to as the dynamic permeability. The theory of dynamic permeability is established \cite{avellaneda1991rigorous,johnson1987theory, auriault1985dynamics-of-por, biot1962mechanics} and further developed by {M.-Y.} Ou \cite{ou2014reconstruction}.

{The goal of this paper is to study how the permeability tensor derived from the homogenization approach for porous {materials} \cite{tartar1980incompressible, tice2014} depend on the microstruture of the pore space}. Details of this will be presented in Section \ref{def_perm}.  

The main tool we use will be the integral representation formula (IRF) for composite materials. Composite materials are materials made from more than one constituent materials with different physical or chemical properties. The effective properties {of composites}, such as elasticity, conductivity and permeability are of great interest in different application fields. {Homogenization theory for composite materials has been} extensively studied in \cite{bensoussan2011asymptotic,milton2002theory, Shamaev_Yosifian_2009}.  Mathematically, for a two-component composite material, the{microstructural} information is carried into the analytical formulation of the effective properties of the composite.  Bergman pioneered the study of analyticity of the effective dielectric constant \cite{bergman1978dielectric}, and in terms of integral representation of effective material properties,  a rigorous basis of integral representation of the effective conductivity is established by {K.} Golden and {G.} Papanicolaou \cite{golden1983bounds}, {the} effective elastic constants by {Y.} Kantor and {D.} Bergman \cite{kantor1982elastostatic}, {the} effective diffusivity in convection-enhanced diffusion was derived by {M.} Avellaneda and {A.} Majda \cite{avellaneda1989stieltjes,avellaneda1991integral}. Further enlargements of the domain of analyticity of the IRF of elasticity tensor to the {case} where one phase is  {a} void or {a} hard inclusion is studied by Bruno and Leo  \cite{bruno1991effective,bruno1993stiffness}.

Unlike the problem setup for calculating effective material properties such as effective conductivity, elasticity, and diffusivity, where the physical property of interest is well defined both in the micro-scale and the macro(homogenized)-scale, the permeability of porous {material} is by definition an effective property and hence {it} makes sense only in the macro-scale. To overcome this difficulty, we {consider a porous material as the limit case of a two-fluid mixture. }

{Specifically, }we will start with the two-fluid mixture problem studied in \cite{Lipton_Avellaneda_1990}, where the effective property is called the \emph{self-permeability}. We will derive the IRF for the {self-}permeability and show that the permeability for a porous {material} is equal to the limit of the self-permeability when the viscosity of one phase becomes infinite. Similar to the hard/soft inclusion case studied in \cite{bruno1991effective,bruno1993stiffness}, we will extend the domain of analyticity of the IRF to $\infty$ and to $0$ by an iterative process. As a result, the IRF derived here is valid for porous {materials} with a solid skeleton as well as for fluid-bubble mixtures. Hence it provides a theoretical connection between the permeability defined in \cite{tartar1980incompressible} and the self-permeability for {the} bubbly fluid studied in \cite{Lipton_Avellaneda_1990} and any mixture in between these two limiting cases. 

{The paper is organized as follows. The permeability of a porous material is defined in Section \ref{def_perm}.  Section \ref{Section2} starts with the definition of the \emph{self-permeability} $\bK$of a two-fluid mixture and the corresponding cell problem, followed by an analysis of the cell problem and the construction of the solution in the vicinity of the two limiting cases of $z=\infty$ and $z=0$. In Section \ref{Section3}, the IRF of $\bK$ is obtained by applying the theory of matrix-valued Stieltjes functions. In this section, the spectral representation of $\bK$ is also derived. The relationships  between the moments of the measure in the IRF and the geometry of the pore space are derived by comparing these two representations. Section \ref{Section4} presents the numerical solutions of the cell problem of a special pore structure, which validate the theoretical results given in Section \ref{Section3}. }

Einstein summation convention is applied unless stated otherwise.

\subsection{\label{def_perm}Definition of Permeability from Homogenization}
Following the convention of homogenization, the space coordinates for the cell problem in the open unit cell $Q=(0,1)^n$ for $n=2,3$, are denoted by $\by=(y_1, y_2, y_3)$. Let $\Omega$ be a {smooth bounded} open set and $Q$ an open unit cube made of two {open sets} $Q_1$, $Q_2$ { and the interface $\Gamma=\mbox{cl}(Q_1)\cap \mbox{cl}(Q_2)$ with cl($A$) being the closure of a set $A$.}  {Moreover, $\widetilde{Q_i}$ denotes the $Q$-periodic extension of $Q_i$,  $i=1,2$. Following \cite{allaire1989homogenization}, we assume that (1) $Q_1$ and $Q_2$ have strictly positive measures in cl($Q$). (2) The set $\widetilde{Q_i}$ is open with $C^1$ boundary and is locally located on one side of its boundary, $i=1,2$, and $\widetilde{Q_1}$ is connected. (3) $Q_1$ is connected with a Lipschitz boundary. In addition, we  consider the case of inclusion, i.e. $Q_2\cap \partial Q=\emptyset$.}

Consider $\epsilon>0$ much smaller than the size of $\Omega$ and {$\epsilon Q$-}periodically extend $\epsilon Q_1$ in the entire space. $\Omega_\epsilon$ denotes the intersection of $\Omega$ and this  {$\epsilon Q$-}periodically extended structure. In \cite{tartar1980incompressible}, the permeability is derived from the Stokes equation in $\Omega_\epsilon$, which reads: find $\bu^\epsilon\in H_0^1(\Omega_\epsilon)^n$ and $p^\epsilon \in L^2(\Omega_\epsilon)/\mathbb{R}$ such that 
\begin{equation}
\left\{ \begin{split}
-\mu \triangle \bu^\epsilon + \nabla p^{\epsilon}  &= \f \quad \text{ in } \Omega_\epsilon. \\
\text{div} \bu^{\epsilon} &= 0 \quad \text{ in } \Omega_\epsilon\\
\end{split}
\right. 
\label{Tartar_stokes}
\end{equation}
where $\f\in L^2(\Omega)$ is independent of $\epsilon$ and the viscosity $\mu$ is a constant ($\mu$ is set to 1 in \cite{tartar1980incompressible, tice2014}). See Figure \ref{cell} for an example of the unit cube. {Note that the superscript $\epsilon$ is used to signify that the solutions $\bu^\epsilon$ and $p^\epsilon$ depend on $\epsilon$.}
\begin{figure}[h]
    \centering
    \includegraphics[width=3cm, height=3cm]{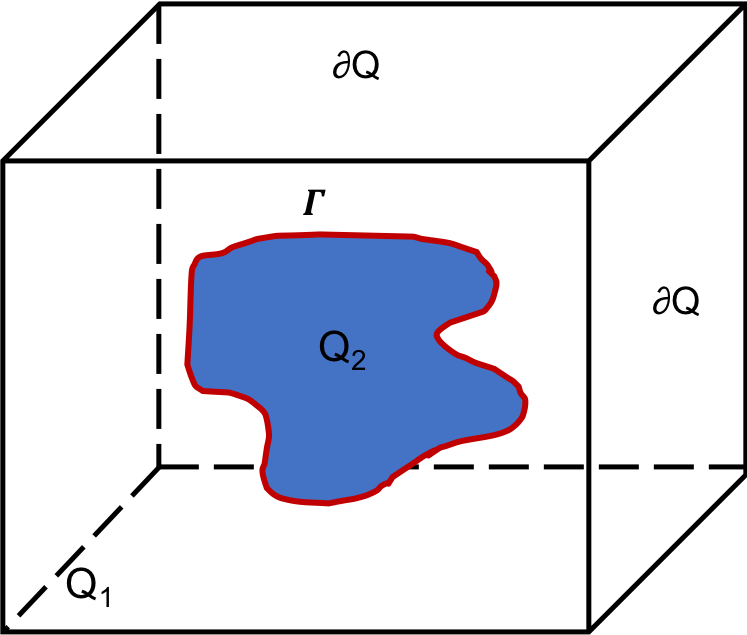}
    \caption{A sample illustration of a periodic cell.}
    \label{cell}
\end{figure}
To be able to prove {the} convergence of $(\bu^\epsilon,p^{\epsilon})$ as $\epsilon\rightarrow 0$, it is necessary to extend these solutions from $\Omega_\epsilon$ to $\Omega$ so they are defined in the same spatial domain.  In \cite{tartar1980incompressible, tice2014}, $\bu^\epsilon$ was extended by zero and $p^\epsilon$ by a properly defined extension operator with their extensions denoted by $\hat{\bu^\epsilon}$ and $\hat{p^\epsilon}$, respectively. As $\epsilon\rightarrow 0$, 
\[
\frac{\hat{\bu^\epsilon}}{\epsilon^2} \rightharpoonup \bU \mbox{ weakly in }L^2(\Omega)^n, \, \text{div}\, \bU=0, \bU\cdot \bn=0 \mbox{ on } \Gamma, \mbox{ and }  \hat{p^\epsilon} \rightarrow p \mbox{ in } L^2(\Omega)/\mathbb{R}
\]
and the limit functions {satisfy} the following Darcy's law \cite{tartar1980incompressible}
\begin{equation}
    \bU = \frac{\bK^{(D)}}{\mu}(\f -\nabla p)
\end{equation}
where the permeability tensor $\bK^{(D)}$ is defined as 
\begin{equation}
    K_{ij}^{(D)} =   \int_{Q_1} \bu_D^{j}\cdot \be_i d\by,\, i,j=1,\cdots n
    \label{K_def}
\end{equation}
with $\be_i$ denoting the unit vector in the $i$-th direction and  $\bu_D^{j}$ the unique solution of the following {boundary value problem}
\begin{equation}
\left\{
\begin{split}
	\mu\Delta_{\by} \bu_D^{j} - \nabla_{\by}p^{j}   = -\be_j  \quad &\text{in } Q_1\\
	\text{div}_{\by} \bu_D^{j} = 0 \quad &\text{in } Q_1\\
	\bu_D^{j} = \bo \quad &\text{on } \Gamma
\end{split}
\label{classical_cell_prob}	
\right.
\end{equation}
in the space $ \mathring{H}(Q_1): = \left\lbrace \bv: \bv\in H^1(Q_1)^n\biggr\vert \; \text{div}_{\by} \bv = 0, \bv\vert_{\Gamma} = \bo, Q \text{-periodic}  \right\rbrace$. {Note that the subscript $i$ of $\bu$ and $p$ signifies the solutions corresponding to the force term $\be_i$.}

Since $\mu$ is set to 1 in \cite{tartar1980incompressible, tice2014}, the permeability $\bK$ presented there is related to $\bK^{(D)}$ by $\bK^{(D)}={\bK}/{\mu}$.
For future analysis, we will derive here the quadratic form representation of the permeability.  We start by observing that for incompressible fluid, we have 
\[
\triangle \bu= \mbox{div} (\nabla \bu+\nabla^T \bu)
\]
Therefore, \eqref{K_def} can be expressed as
\begin{equation}
    K_{ij}^{(D)} =    \int_{Q_1} \bu^{j}_{{D}}\cdot \left( \nabla_{\by}p^{i} - {\mu}\Delta_{\by} \bu^{i}_{{D}} \right) d\by = \int_{Q_1} {\mu}\nabla_{\by} \bu^{j}_{{D}} : \left(\nabla_{\by} \bu^{i}_{{D}}+\nabla_\by^T \bu^i_{{D}}\right) d\by
    \label{K_tartar_form}
\end{equation}
after applying Divergence theorem, periodicity of $\bu$, $p$ and no-slip conditions on $\Gamma$. Here we have used the Frobenius inner product of matrices $\bA : \bB  = \sum_{i,j = 1} A_{ij}B_{ij}$.
In terms of the usual notion of the symmetric part and the antisymmetric part of vector field $\nabla\bu$
\begin{eqnarray}
e(\bu) := \frac{1}{2} \left(\nabla \bu + \nabla^T \bu\right)\label{sym},\,\,\tilde{e}(\bu):= \frac{1}{2} \left(\nabla \bu - \nabla^T \bu\right){,}\label{antisym}
\end{eqnarray}
the right-hand side of equation \eqref{K_tartar_form} becomes $\int_{Q_1} 2\mu (e(\bu_D^j)+\tilde{e}(\bu_D^j)) : e( \bu_D^i )d\by =\int_{Q_1} 2\mu e(\bu_D^j) : e(\bu_D^i) d\by$
because the Frobenius product of a symmetric matrix and an antisymmetric matrix must be 0. Therefore we have the quadratic form of permeability tensor $\bK^{(D)}$ 
\begin{equation}
    K_{ij}^{(D)} =\int_{Q_1} 2\mu e(\bu_D^j) : e(\bu_D^i) d\by.
    \label{quadratic_form_permeability}
\end{equation}
%
%
\section{\label{two-fluid}Approximation of flow in porous medium by a two-phase Stokes flow}
\label{Section2}
In this section, we consider the system for porous {materials} \eqref{classical_cell_prob} as one of the limiting cases of the two-fluid problem described below, which is the same as the one studied in \cite{Lipton_Avellaneda_1990} with the exception that the fluid viscosity here can be complex-valued. It is easy to check that the homogenization process in \cite{Lipton_Avellaneda_1990} stays valid after making small modifications to accommodate the complex valued viscosity described below.  

Let $\Omega$, $Q$ and $\epsilon$ be the same as in Section \ref{def_perm}. $Q_2$ is still the inclusion in the periodic cell. Consider the ${\epsilon Q}$-periodic extension of $\epsilon Q_1$ (${\epsilon}Q_2$) and denote by $\Omega_{1{\epsilon}}$ ($\Omega_{2\epsilon}$) its intersection with $\Omega$. We note that $\Omega_{1{\epsilon}}$ (region of the hosting fluid)  is the same as $\Omega_\epsilon$ in the previous section. Suppose $\Omega_{1{\epsilon}}$ is occupied by fluid with viscosity $\mu_1>0$ and $\Omega_2$ by fluid with viscosity $z\mu_1$ with $z\in \C$. The interface $\tilde{\Gamma} = \partial \Omega_{1{\epsilon}} \cap  \partial \Omega_{2{\epsilon}}$ is such that $\Omega_{1{\epsilon}}\cup  {\tilde{\Gamma}\ \cup\ }{\Omega_{2{\epsilon}}} = \Omega$. 
%
%
{ For the ease of notation, we define the stress tensor $\boldsymbol{\tau}(\bu,\bmu)$ of a fluid with viscosity $\mu$, velocity field $\bu$ and pressure field $p$ as 
\begin{equation}
	\boldsymbol{\tau}(\bu,p, \mu) = 2{\mu} e(\bu) - p\B{I}, \quad \mbox{$\B{I}$ is the identity matrix.}
	\label{def_total_stress}
\end{equation}
}
{Let $\chi_i$ be the characteristic function of $\Omega_{\epsilon i}$, $i=1,2$, consider the viscosity function}
\beq
{\xi}^{{\epsilon}}(\bx;z) = (\chi_2(\bx) z \mu_1 + \chi_1(\bx) \mu_1), \quad z \in \C.
\label{piecewise_mu}
\eeq
The two-fluid problem is {given by} the following Stokes system
\begin{equation}
\left\{ \begin{split}
\text{div} \left( 2{{\xi}}^{\epsilon}(\bx;z)e(\bu^{\epsilon}) \right) - \nabla p^{\epsilon}  &= -\f \quad \text{ in } \Omega\backslash \tilde{\Gamma} \label{global_first_form_std_notation}\\
\text{div} \bu^{\epsilon} &= 0 \quad \text{ in } \Omega\\
\bu^{\epsilon} &= \bo \quad \text{ on } \partial\Omega\\
\llbracket \bu^{\epsilon} \rrbracket=0,\,\bu^{\epsilon} \cdot \bn &= 0 \quad \text{on } \tilde{\Gamma}\\
\llbracket \bpi \rrbracket\cdot \bn  = \left( \llbracket \bpi \cdot \bn\rrbracket \cdot \bn \right) \bn &\equiv\llbracket \bpi\cdot \bn\rrbracket- \bn \times \bn \times \llbracket \bpi\cdot \bn\rrbracket
\quad \text{on } \tilde{\Gamma}
\end{split}
\right. 
\end{equation}
where {$\bpi=\btau(\bu^\epsilon,p^\epsilon, \boldsymbol{\xi}^\epsilon)$}, $\f$ is a square integrable momentum source {independent of $\epsilon$}, $\llbracket\cdot \rrbracket$ the jump across the interface $\Tilde{\Gamma}$, {and $\bn$ is the outward unit normal of $\partial \Omega_{2\epsilon}$.} The second jump condition in \eqref{global_first_form_std_notation} means the traction can only jump in the normal direction. {Also note that the superscript $\epsilon$ is used to signify that the solutions $\bu^\epsilon$ and $p^\epsilon$ depend on $\epsilon$.}

It is shown in \cite{Lipton_Avellaneda_1990} that as $\epsilon \to 0$, $\bu^\epsilon$ and the properly normalized $p^\epsilon$, which is denoted by $\hat{p}^\epsilon$, converge as follows
\begin{equation*}
    \frac{{\bu^{\epsilon}}}{\ep^2} \to \bu^0 \quad \text{weakly in }L^2(\Omega)^n,\;\hat{p^{\ep}} \to P \quad \text{strongly in }L^2(\Omega)/\R
\end{equation*}
where $\bu^0$ and $P$ satisfy the homogenized system:
\begin{equation}
    \left\lbrace\begin{split}
        \bu^0 &= -\bK(\nabla P - \B{f})\quad \text{in }\Omega \label{self_K_def}\\
        \text{div } \bu^0 &= 0\quad \text{in }\Omega\\
    \end{split}\right.\end{equation}
where the components of $\bK$, {which is referred to as the \emph{self-permeability} in \cite{Lipton_Avellaneda_1990}}, is defined as
\begin{equation}
K_{ij}(z) \defeq  \int_{Q} {\bu^j}\cdot \be_i d\by,\, {i,j=1,\dots,n}
\label{permeability_def}
\end{equation}
with $\bu^{{i}}$ being the unique solution to the cell problem posed in the function space $H(Q)$, which is defined in \eqref{space_HQ},
\begin{equation}
\left \{ \begin{split} \text{div}_{\by} \left( 2{\bmu}(\by;z) e( \bu^i) - p^i \bI \right)  + \be_i &= \bo \quad \text{in } Q_1 \cup Q_2\\
\llbracket \bpi \rrbracket\cdot\bn  &= \left( \llbracket \bpi\cdot \bn\rrbracket \cdot \bn \right) \bn \text{ on } \Gamma
\end{split}\right.
\label{cell_prob}
\end{equation}
{where $\bmu(\by;z)=\mu_1 \chi_1(\by)+z\mu_1 \chi_2(\by)$ with $\chi_m$ being the characteristic functions of $Q_m$, $m=1,2,$ and $\bpi=\btau(\bu^k,p^k,\bmu)$, {cf. \eqref{def_total_stress}}. Note that the superscript $i$ is used to signify that $\bu^i$ and $p^i$ are solutions to the cell problem \eqref{cell_prob} with the force term $-\be_i$, $i=1,\dots,n$. }
\subsection{\label{C_pro}Function Spaces}
Let $\mathcal{R}(Q_2)$ denote the space of rigid body displacement{s} in $Q_2$, i.e. $\bu = \bA\by + \bb$ with constant skew-symmetric matrix $A$ and constant vector $\bb$ in $Q_2$. We start with the space of admissible functions for the velocity
\begin{align}
H(Q) &:= \left\{ \bv: \bv\in H^1(Q_1\cup Q_2)^n \biggr\vert \; \text{div}_{\by} \bv = 0,\; \bv \cdot \bn = 0 \text{ in } H^{-\frac{1}{2}}(\Gamma),\right.\nonumber \\
&\qquad \left.{} \llbracket \bv \rrbracket_{\Gamma} = \bo,\; (\bv,\B{\eta})_{H^1(Q_2)}=0, \forall \B{\eta}\in \mathcal{R}(Q_2),\,\bv \text{ is } Q \text{- periodic}  \right\} 
\label{space_HQ}
\end{align}
{where $\bn$ is the outward unit normal of $\partial Q_2$. $H
(Q)$ is} endowed with {the} inner product
\begin{equation}
(\bu,\bv)_{Q} = \int_{Q} 2\mu_1e(\bu) : \overline{e(\bv)} d\by.
\label{inner_product_def}
\end{equation}
The induced norm is denoted by $\norm{\bu}_{Q}^2 :=  (\bu,\bu)_Q$. Note that we have $H(Q)\cap \mathcal{R}(Q)=\{\B{0}\}$ because $\bA = 0$ due to the $Q$-periodicity and $\bu \cdot \bn = 0$ implies $\bb = \bo$. 
We observe that if $\bu \in H(Q)$ then $\bu \in H^1(Q)^n$ {by the following argument}. Obviously, $\bu \in L^2(Q)^n$. To prove $\frac{\partial u_i}{\partial y_j} \in L^2(Q)$ for $i,j = 1,2,3$, let $\phi$ be any {$C^\infty$} test function compactly supported in $Q$ and $h$ be the $i$-th component $u_i$ for any $i$. Then
\begin{equation*}
    \int_{Q} h \nabla \phi d\by = - \left( \int_{Q_1\cap \text{Supp}(\phi)} \phi\nabla h d\by + \int_{Q_2\cap \text{Supp}(\phi)} \phi\nabla h d\by\right)
\end{equation*}
here we used $\llbracket h \rrbracket = \bo$. Now we can define a candidate function $\bg$ such that
\begin{equation}
    \begin{split}
        \bg\vert_{Q_i} &:= \nabla h\vert_{Q_i},\,i=1,2
    \end{split}
\end{equation}
then {clearly} $\bg\in L^2(Q)^n$ {and} $\left< h,\nabla \phi \right> = -\left< g,\phi \right>$, where $\left<\cdot \right>$ denotes the usual $L^2$ inner product. Therefore  ${h} \in H^1(Q)$ {and hence $u_i \in H^1(Q), \ i=1,\dots,n$}.

Next, we show that $\| \cdot \|_Q$ is equivalent to the usual $H^1$ norm, i.e.{,} there exist constants $B_1$ and $B_2$ such that
\begin{equation}
    B_1 \norm{\bu}_{H^1(Q)} \leq \norm{\bu}_{Q} \leq B_2 \norm{\bu}_{H^1(Q)}
    \label{equi_norm}
\end{equation}

Because $H^1(Q)\cap \mathcal{R}(Q)=\{0\}$, by Theorem 2.5 in \cite{Shamaev_Yosifian_2009}, there exists a Korn's constant $C_1$ such that
\begin{equation}
    C_1 \norm{\bu}_{H^1(Q)} \leq \frac{1}{\sqrt{2\mu_1}} \norm{\bu}_{Q}
\end{equation}
where $C_1$ depends only on $Q$. Therefore, we can take $B_1=\sqrt{2\mu_1}C_1$. To emphasize the dependence on $Q$, we will write it as $B_1(Q)$. On the other hand, according to the orthogonal decomposition that $\nabla \bu = e(\bu) + \tilde{e}(\bu)$, see \eqref{sym},
$$ \norm{\bu}_{H^1(Q)}^2 \geq \norm{\nabla \bu}_{L^2(Q)}^2 = \norm{e(\bu)}_{L^2(Q)}^2 + \norm{\tilde{e}(\bu)}^2 \geq \norm{e(\bu)}_{L^2(Q)}^2 =\frac{1}{2\mu_1} \|\bu \|_Q^2 $$
therefore $B_2 = \sqrt{2\mu_1}$. The reason for introducing the $H(Q)$-norm is that the self-permeability in  \eqref{permeability_def} can be represented in terms of the inner product. More specifically, {using  \eqref{permeability_def},  \eqref{cell_prob}  and the fact that $\overline{\be_i}=\be_i$}, by a calculation similar to \eqref{K_tartar_form} and taking into account the interface condition $\bu\cdot \bn=0$ and the jump conditions in \eqref{cell_prob}, \eqref{permeability_def} can be expressed in the following form
\begin{equation}
\begin{split}
    K_{ij}(z) 
 =  \int_{Q}2  {\mu}(\by;\overline{z}) \overline{e( \bu^{{i}}(z))} : {e(\bu^{{j}}(z))}  d\by
\end{split} \label{permeability_second_form}  
\end{equation}
and its conjugate transpose $\bK^*:=\overline{\bK^T}$ is
\begin{equation}
    {(K^*)_{ij}}(z) 
    =  \int_{Q}2  {\mu}(\by;{z}) {e( \bu^{{j}}(z))} : \overline{e(\bu^{{i}}(z))}  d\by
\end{equation}
\subsection{Weak solution of the Cell Problem \eqref{cell_prob} }
The weak formulation of the cell problem \eqref{cell_prob} is
\begin{equation}
\int_{Q_1\cup Q_2} 2\mu(\by;z) e( \bu^k) : \overline{e(\bv)} d\by = \int_{Q_1\cup Q_2} \be_k\cdot \bar{\bv} d\by ,\qquad \forall \bv \in H(Q)
\label{variational_form_cell_prob}
\end{equation}
From this, we see that the solutions satisfy the following symmetry
\[
\bu^k(\by;\overline{z})=\overline{\bu^k(\by;z)}
\]

Define the sesquilinear form on $H(Q)$
\begin{equation}
a(\bu,\bv) = \int_{Q_1\cup Q_2}2\mu(\by;z) e( \bu^k) : \overline{e(\bv)}  d\by
\end{equation}
It is clear that $a(\bu,\bv)$ is bounded in $H(Q)$. To check the coercivity, assume $\bu^k \neq 0$ and define the parameter
\begin{equation}
\lambda := \frac{\int_{Q}2 \mu_1\chi_2 e( \bu^k) : \overline{e(\bu^k)} d\by}{\int_{Q}2 \mu_1 e( \bu^k) : \overline{e(\bu^k)} d\by}
\end{equation}
then $0\leq \lambda \leq 1$. We note that  
\begin{eqnarray}
\frac{a(\bu^k,\bu^k) }{\int_{Q}2\mu_1 e( \bu^k) : \overline{e(\bu^k)} d\by}= \lambda z+(1-\lambda)\cdot 1%
\label{coercivity_original_problem}
\end{eqnarray}
and hence as long as  0 is not on the line segment joining $z$ and 1, there exist $\alpha(z):=\min_{0\le \lambda\le 1} |\lambda z+1-\lambda|>0$ such that 
\begin{equation}
\left|a(\bu^k,\bu^k)\right|\ge \alpha(z) \int_{Q}2\mu_1 e( \bu^k) : \overline{e(\bu^k)} d\by=\alpha\| \bu^k\|_Q^2
\end{equation}

Therefore for $z\in \C\backslash\left\lbrace \Re{z}\leq0\right\rbrace $, by {the} Lax-Milgram Lemma \cite[Chapter~2]{brenner2007mathematical}, there exists a unique weak solution $\bu^k\in H(Q)$ to the cell problem \eqref{cell_prob} and with the solution $\bu^k$, we can construct $p^k \in L^2(Q)/ \mathbb{C}$. 

Since $\alpha(z)$ is a continuous function in $z$, the coercivity of the sesquilinear form can be applied to conclude that $\bu^k$ is analytic in $z$ and its $m$-th derivative, $m\ge1$, satisfies the following recursive equation
\begin{equation}
\int_{Q_1\cup Q_2} 2\mu(\by;z) e\left( \frac{d^m\bu^k}{dz^m}\right) : \overline{e(\bv)} d\by = -\int_{Q_2} 2 {m}\mu_1 e\left(\frac{d^{m-1}\bu^k}{dz^{m-1}}\right):e(\overline{\bv}) d\by,\,\forall \bv \in H(Q)
\end{equation} 
As a result, $\bK(z)$ is also analytic for $z\in \C\backslash\left\lbrace \Re{z}\leq0\right\rbrace$. 
To relate the two-fluid problem with $\bK^{(D)}$, we adapt the method used in \cite{bruno1993stiffness} to study the behavior of $\bK(z)$ near $z=\infty$ in the following section.
\subsection{\label{large_z} Analyticity of the Solution for large $|z|$}
Let $w:=\frac{1}{z}$ and consider $Q$-periodic solution in the series form near  $w=0$
\begin{equation}
\bu_\infty(\by;\be,w) := \sum_{k = 0}^{\infty} \bu_k(\by;\be) w^k \mbox{ and } p_\infty(\by;\be,w) := \sum_{k = 0}^{\infty} p_k(\by;\be) w^k
\label{large_ansaz}
\end{equation}
where the $\be$ is an arbitrary constant unit vector. To set up the notation, we denote the restrictions of $\bu_k$, $p_k$ in $Q_2$ (inclusion) and $Q_1$ as $\bu^{in}_k$, $p^{in}_k$  and $\bu^{out}_k$, $p^{out}_k$ respectively and define 
\begin{align}
&\bu_\infty^{in}(\by;\be,w): = \sum_{k = 0}^{\infty} \bu^{in}_k(\by,\be) w^k,\qquad
\bu_\infty^{out}(\by;\be,w): = \sum_{k = 0}^{\infty} \bu^{out}_k(\by,\be) w^k \label{u_out}
\end{align}
By substituting \eqref{large_ansaz} into \eqref{cell_prob} with the viscosity defined in \eqref{piecewise_mu}, taking into account the additional two interface conditions $\bu\cdot\bn=0$ and $\llbracket \bu \rrbracket=0$, followed by equating terms of the same order with respect to $w$,  we arrive in the following equations
in $Q_1$:
	\beqa
	&& O(w^0):\qquad
	\text{div}_{\by} \left( 2\mu_1 e(\bu_0^{out}) - p_0^{out}\bI \right)  = -\be
	\label{order_0_Q1}\\
&&O(w^k):\qquad
	\text{div}_{\by} \left( 2\mu_1 e(\bu_k^{out}) - p_k^{out}\bI \right)  = \bo  \mbox{ for } k\geq 1
	\label{order_k_Q1}
	\eeqa
and in $Q_2$:
	\beqa
	&&O(w^{-1}):\qquad \text{div}_{\by} \left( 2\mu_1 e(\bu_0^{in}) \right)  = \bo
	\label{order_m1_pde_Q2}
\\
&&O(w^{0}):\qquad	\text{div}_{\by} \left( 2\mu_1 e(\bu_1^{in}) - p_0^{in}\bI \right)  = -\be 
	\label{order_0_Q2}\\
&&O(w^{k}):\qquad
\text{div}_{\by} \left( 2\mu_1 e(\bu_{k+1}^{in}) - p_k^{in}\bI \right)  = \bo \mbox{ for } k\geq 1
	\label{order_k_Q2}
\eeqa
and the following interface conditions on $\Gamma$
\beqa
O(w^{-1}):&&\ 
	2\mu_1(e(\bu_0^{in})\cdot \bn)|_\Gamma = C(\by) \bn \text{ for some function }C(\by) 
	\label{order_m1_interface_cond}\\
 O(w^k),\, k\ge 0:&&\ 
	 \left( \left( 2\mu_1 e(\bu_k^{out}) - p_k^{out}\bI \right) - \left( 2\mu_1 e(\bu_{k+1}^{in}) - p_k^{in}\bI \right)  \right) \bn \label{order_k_interface_cond}  \\
	&&  = \left\{ \left[\left( \left( 2\mu_1 e(\bu_k^{out}) - p_k^{out}\bI \right) - \left( 2\mu_1 e(\bu_k^{in}) - p_k^{in}\bI \right)  \right) \bn \right]\cdot \bn\right\} \bn,
\nonumber	\\
&&\bu^{in}_k \cdot \bn = \bu^{out}_k \cdot \bn = 0 \mbox{ and }\bu^{in}_k = \bu^{out}_k
	\label{interface_cond_3}
\eeqa
We introduce the following spaces, $i=1,2$
\begin{eqnarray*}
&&H(Q_1) = \left\lbrace \bv: \bv\in H^1(Q_1)^n\biggr\vert \; \text{div}_{\by} \bv = 0, \bv \cdot \bn = 0 \text{ on } \Gamma, Q \text{-periodic}  \right\rbrace \label{HQ_1}\\
&&H(Q_2) = \left\lbrace \bv: \bv\in H^1(Q_2)^n\biggr\vert \; \text{div}_{\by} \bv = 0, \bv \cdot \bn = 0 \text{ on } \Gamma, \, (\bv,\mathcal{R}(Q_2))_{H^1(Q_2)}=0\right.\\ &&\left.\qquad \qquad , Q \text{-periodic}  \right\rbrace \label{HQ_2}\\
&&\mathring{H}(Q_i)= \left\lbrace \bv: \bv\in H^1(Q_i)^n\biggr\vert \; \text{div}_{\by} \bv = 0, \bv\vert_{\Gamma} = \bo, Q \text{-periodic}  \right\rbrace \label{Hring_i} \subset H(Q_i),\,\\
&&L(Q_i)/ \mathbb{C} = \left\lbrace p: p\in L^2(Q_i), \int_{Q_i} p(\by)d\by = 0, Q \text{-periodic},\right\rbrace
\end{eqnarray*}
Note that $H(Q_1)\cap \mathcal{R}(Q_1)=\{\bo\}$ because $\partial Q \subset \partial Q_1$. For $H(Q_2)$, the boundary condition $\bu\cdot\bn=0$ implies $H(Q_2)\cap \mathcal{R}(Q_2)=\{\bo\}$ because of the extra condition $(\bv,\mathcal{R}(Q_2))_{H^1(Q_2)}=0$\cite{Shamaev_Yosifian_2009}.
Therefore, $H(Q_i)$ and $\mathring{H}(Q_i)$ are equipped with inner product $(\bu,\bv)_{Q_i} = \int_{Q_i} 2\mu_1e(\bu) : \overline{e(\bv)} d\by$ and Korn's inequalities are valid in $H(Q_i)$, $i=1,2$.
\begin{lemma}
	Let $Q_2$ be a connected, open bounded set such that $\pr Q_2 \cap \pr Q=\emptyset$ and  $\pr Q_2$ is in $\c^{k,\sigma}$ {, $k,\sigma\ge 0$, $k+\sigma\ge2$}. For any vector field $\bu^{in} \in H(Q_2)$, there exists a unique weak solution $\bu^{out}(\by;\B{f}^{out})\in H(Q_1)$ that satisfies the following system 
	\begin{equation}
	\left\lbrace
	\begin{split}
	    \text{div}_{\by} \left( 2\mu_1 e(\bu^{out}) - p^{out}\bI \right) &= \mathbf{f}^{out} \quad \text{ in } Q_1\\
	    \bu^{out} &= \bu^{in} \quad \text{ on }\Gamma
	\end{split}
	\right.
	\label{lemma_1_eqn}
	\end{equation}
	where in our context, $\mathbf{f}^{out} = \bo \text{ or } \mathbf{f}^{out} = -\be$, a constant unit vector. Moreover, 
	\beq
 \norm{\bu^{out}}_{Q_1} \leq \frac{1}{B_1(Q)} \norm{\mathbf{f}^{out}}_{L^2(Q_1)} + 2E_1 \norm{\bu^{in}}_{Q_2}.	
	\eeq
	where the positive constants $B_1(Q_{{1}})$ is defined in \eqref{equi_norm} and $E_1\ge 1$ depends only on $Q_1$ and $Q_2$.
	\label{lemma_1} 
\end{lemma}

\begin{proof}  
\quad To handle the inhomogeneous boundary condition, we proceed as follows. By  \cite[Corollary 3.2]{Kato_Mitrea_Ponce_Taylor_2000}, there exists a bounded, divergence free extension $T\left(\bu^{in} \right) $ of $\bu^{in}$ to a small neighborhood $O$ of $Q_2$ and vanishes at $\partial O\subset Q_1$ such that 
    	\begin{equation}
    		\norm{T\left(\bu^{in} \right)}_{Q} \leq E_1 \norm{\bu^{in}}_{Q_2}
		\label{ext_operator}
    	\end{equation}
    	where $E_1\ge 1$ depends only on $Q_1$ and $Q_2$. Furthermore, the extension $T\left(\bu^{in} \right)$ on $O$ can be extended periodically to  $\R^n$ \cite{conca1985application} since $T\left(\bu^{in} \right)$ vanishes on $\partial O$ and hence on $\partial Q$. We denote the restriction $T(\bu^{in})|_{Q_1}$ as $\tilde{\bu}^{out}\in H(Q_1)$ and $\mathring{\bu}^{out} \defeq \bu^{out} - \tilde{\bu}^{out} \in \mathring{H}(Q_1)$ and \eqref{lemma_1_eqn} becomes
    	\begin{equation}
    		\text{div}_{\by} \left( 2\mu_1 e(\mathring{\bu}^{out}) - p^{out}\bI \right)  = \mathbf{f}^{out} - \mu_1\Delta \tilde{\bu}^{out} \quad \text{ in } Q_1
    		\label{lemma_1_eqn2}
    	\end{equation}
Consider the variational formulation: Find $\mathring{\bu}^{out} \in \mathring{H}(Q_1)$ such that $ \forall \Phi \in \mathring{H}(Q_1)$, 
    	\begin{equation}
    		\int_{Q_1} 2\mu_1 e( \mathring{\bu}^{out}) : \overline{e(\Phi)} d\by = -\int_{Q_1} \mathbf{f}^{out} \cdot \overline{\Phi} d\by - \int_{Q_1} 2\mu_1 e( \tilde{\bu}^{out}) : \overline{e(\Phi)} d\by, \quad
    		\label{variational_form_for_in_to_out}
	    \end{equation}
The right hand side of \eqref{variational_form_for_in_to_out} can be bounded as follows
\begin{eqnarray*}
    &\left|  \int_{Q_1} \mathbf{f}^{out} \cdot \overline{\Phi} d\by + \int_{Q_1} 2\mu_1 e( \tilde{\bu}^{out}) : \overline{e(\Phi)} d\by\right| 
	    \leq& \left(\frac{\norm{\mathbf{f}^{out}}_{L^2(Q_1)}}{B_1(Q_{{1}})} + \norm{\tilde{\bu}^{out}}_{Q_1} \right) \norm{\Phi}_{Q_1} \\
\end{eqnarray*}
The sesquilinear form  $\int_{Q_1} 2\mu_1 e( \mathring{\bu}^{out}) : \overline{e(\Phi)} d\by$ is clearly bounded and coercive with constant 1. Hence by {the} Lax-Milgram Lemmat there exists a unique weak solution $\mathring{\bu}^{out}\in \mathring{H}(Q_1)$ to \eqref{lemma_1_eqn2} such that $ \norm{\mathring{\bu}^{out}}_{Q_1}\le (\frac{1}{B_1(Q_{{1}})}\norm{\mathbf{f}^{out}}_{L^2(Q_1)} + \norm{\tilde{\bu}^{out}}_{Q_1})$.  In terms of $\mathring{\bu}^{out}$, $\bu^{out}$ can be expressed as $ \bu^{out} = \tilde{\bu}^{out} + \mathring{\bu}^{out}$ and satisfies the estimate
	\begin{equation}
	    \norm{\bu^{out}}_{Q_1} \leq \norm{\mathring{\bu}^{out}}_{Q_1} + \norm{\tilde{\bu}^{out}}_{Q_1}\leq \frac{1}{B_1(Q_{{1}})} \norm{\mathbf{f}^{out}}_{L^2(Q_1)} + 2E_1 \norm{\bu^{in}}_{Q_2}
	    \label{lemma_1_boundness}
	\end{equation}
	To show the solution $\bu^{out}$ is unique, suppose $\bu^{out}_1$ and $\bu^{out}_2$ both solve \eqref{lemma_1_eqn} then the difference $\bw^{\text{diff}} = \bu^{out}_1 - \bu^{out}_2 \in \mathring{H}(Q_1)$ must satisfy 
	\[
	    \int_{Q_1} 2\mu_1 e( \bw^{\text{diff}}) : \overline{e(\Phi)} d\by = 0, \qquad \forall \Phi \in \mathring{H}(Q_1)
	\]
Hence $\bw^{\text{diff}}=0$ in $Q_1$ {because  $\bw\in \mathring{H}(Q_1) $}. We note the two special cases:
\beqa
	\norm{\bu^{out}}_{Q_1} \leq 2E_1 \norm{\bu^{in}}_{Q_2} \mbox{ for } \mathbf{f}^{out} = \bo
	\label{u_out_bound_0}\\
\norm{\bu^{out}}_{Q_1} \leq \frac{1}{B_1(Q_{{1}})} \sqrt{|Q_1|} + 2E_1 \norm{\bu^{in}}_{Q_2} \mbox{ for }  \mathbf{f}^{out} = -\be_{{j}},\, {j=1,\dots,n,}
\label{u_out_bound_nonzero}
\eeqa
where $|Q_1|$ is the volume of $Q_1$.
 \end{proof}
\begin{lemma}
	Let $Q_2$ satisfy the same assumptions as those in Lemma \ref{lemma_1}. For any pair of $\left( \bu^{out}, p^{out}\right) \in H(Q_1) \times L^2(Q_1)/ \mathbb{C}$ that satisfies \eqref{lemma_1_eqn}, there exists a unique vector $\bu^{in}(\by;\B{f}^{in})\in H(Q_2)$ that satisfies the Stokes equation with continuity of tangential traction on $\Gamma$
	\begin{equation}
	\left\lbrace
	\begin{split}
	    \text{div}_{\by} \left( 2\mu_1 e(\bu^{in}) - p^{in}\bI \right)  &= \B{f}^{in} \quad \text{in }Q_2,\\
	    \bn\times\bn\times\left[\left( \left( 2\mu_1 e(\bu^{out}) - p^{out}\bI \right) - \left( 2\mu_1 e(\bu^{in}) - p^{in}\bI \right)  \right)\cdot\bn\right] &= \bo \quad \text{on }\Gamma,
	\end{split}
	\right.
	\label{lemma_2_eqn}
	\end{equation}
	where in our context, $\mathbf{f}^{in} = \bo \text{ or } \mathbf{f}^{in} = -\be$. Moreover,
	\[
		\norm{\bu^{in}}_{Q_2} \leq   \frac{1}{B_1(Q_2)}\norm{\B{f}^{in}}_{L^2(Q_2)} + \frac{E_1(Q_1)}{B_1(Q_1)}\norm{\B{f}^{out}}_{L^2(Q_1)} + E_1(Q_1)\norm{\bu^{out}}_{Q_1}
\]
and $E_1(Q_1)$, $B_1(Q_1)$  and $B_1(Q_2)$ {depend only on $Q$ and $\Gamma$}.
		
	\label{lemma_2} 
	\end{lemma}
	
	\begin{proof}\quad Take ${\Phi} \in H(Q_2)$, the variation formulation for the PDE is
		\[
		    \int_{Q_2} \left( \text{div}_{\by} \left( 2\mu_1 e(\bu^{in}) - p^{in}\bI \right) \right) \cdot \bar{\Phi} d\by = \int_{Q_2} \B{f}^{in} \cdot \bar{\Phi} d\by.
		\]
	{For the ease of notation, let $\bpi^{in}:=\btau(\bu^{in},\mu_1,p^{in})$ and   $\bpi^{out}:=\btau(\bu^{out},\mu_1,p^{out})$ where the stress function $\btau$ is defined in \eqref{def_total_stress}}.	
	 Applying integration by parts on the left hand side, followed by {an} application of {the} divergence theorem leads to 
		\begin{equation}
		    - \int_{\Gamma} ({  \bpi^{in} \cdot \bar{\Phi}})  \cdot \bn dS - \int_{Q_2} 2\mu_1 e(\bu^{in}) : \overline{e(\Phi)} d\by = \int_{Q_2} \B{f}^{in} \cdot \bar{\Phi} d\by
		    \label{lemma2_pre_variational_form}
		\end{equation}
		Let $\B{t}$ denote the unit vector in the tangent plane such that $\B{t}\cdot\bn = 0$. The conditions on $\Gamma$ in \eqref{lemma_2_eqn} imply
		\begin{equation*}
		    \left\{
		    \begin{split}
		        &\Phi\cdot\bn = 0\Rightarrow \Phi = d(\by) \B{t} \mbox{ for some function }d(\by),\\ &\bn\times\bn\times\left[\left( \bpi^{out} - \bpi^{in} \right)\cdot\bn\right] = \bo \Rightarrow \left( \bpi^{out} - \bpi^{in} \right)\cdot\bn = C(\by)\bn \mbox{ for some function }C(\by)
		    \end{split}
		    \right.
		\end{equation*}
		With these observations {and the fact that $\bpi^{in}$ is symmetric}, the first term in \eqref{lemma2_pre_variational_form} can be expressed as 
		\begin{align*}
		    - \int_{\Gamma} {  \bar{\Phi}  \cdot \bpi^{in} }\cdot \bn \,dS &= \int_{\Gamma} \overline{d(\by)}\B{t} \cdot \left[\left( \bpi^{out} - \bpi^{in} \right)\cdot\bn\right] - \int_{\Gamma} \overline{d(\by)}\B{t} \cdot (\bpi^{out} \cdot\bn)\, dS\\
		    =& - \int_{\Gamma} \bar{\Phi} \cdot \bpi^{out} \cdot\bn \,dS
		\end{align*}
		and hence the variational form \eqref{lemma2_pre_variational_form} becomes for all $\Phi\in H(Q_2)$
		\begin{equation}
		- \int_{Q_2}2\mu_1 e(\bu^{in}) : \overline{e(\Phi)}d\by  = \int_{Q_2} \B{f}^{in}\cdot \bar{\Phi}d\by + \int_{\Gamma} \bar{\Phi} \cdot \bpi^{out} \cdot\bn dS		
		\label{lemma_2_variational_form}
		\end{equation}
		To bound the right hand side of \eqref{lemma_2_variational_form}, we first extend $\Phi\in H(Q_2)$ by the operator $T$ described in \eqref{ext_operator}
		\begin{equation}
		\norm{T(\Phi)}_{Q} \leq E_1\norm{\Phi}_{Q_2}
		\end{equation}
		$T(\Phi)$ rapidly decays to zero in a small neighborhood of $Q_2$ and stays $0$ for the rest of $Q_1$. The restriction of $T(\Phi)$ in $Q_1$, denoted by $\Phi^{out}$, has the following estimate
		\begin{equation}
		\norm{\Phi^{out}}_{Q_1} \leq \norm{T(\Phi)}_{Q}\leq E_1 \norm{\Phi}_{Q_2}
		\label{bound_for_phi_from_in_to_out}
		\end{equation}
		Hence
		\begin{equation}\nonumber
		\begin{split}
		&\left| \int_{\Gamma} \bar{\Phi} \cdot \bpi^{out} \cdot\bn dS\right|=   \left|  \int_{\Gamma} \bar{\Phi}^{out} \cdot \bpi^{out} \cdot\bn dS \right|\\
		=& \left| \int_{Q_1} \left[ \text{div}_{\by} \left( 2\mu_1 e(\bu^{out}) - p^{out}\bI \right) \right] \cdot \overline{\Phi^{out}} d\by + \int_{Q_1} 2\mu_1 e(\bu^{out}): \overline{e(\Phi^{out})} d\by \right|\\
		\leq& \left| \int_{Q_1}  \B{f}^{out}  \cdot \overline{\Phi^{out}} d\by\right| + \left|\int_{Q_1} 2\mu_1 e(\bu^{out}): \overline{e(\Phi^{out})} d\by \right|\\
		\leq& \frac{E_1}{B_1(Q_{{1}})}\norm{\B{f}^{out}}_{L^2(Q_1)}\norm{\Phi}_{Q_2} + E_1\norm{\bu^{out}}_{Q_1}\norm{\Phi}_{Q_2}
		\end{split}
		\end{equation}
		Therefore the right hand side of \eqref{lemma_2_variational_form} is bounded by
\[
	 \norm{\Phi}_{Q_2} \left( \frac{E_1 \norm{\B{f}^{in}}_{L^2(Q_2)}}{B_1(Q_{{2}})} + \frac{E_1 \norm{\B{f}^{out}}_{L^2(Q_1)}}{B_1(Q_{{1}})} + E_1\norm{\bu^{out}}_{Q_1} \right)
\]
Finally, by the Lax-Milgram Lemma a unique solution $\bu^{in}$ exists such that
		\begin{equation}
		\norm{\bu^{in}}_{Q_2} \leq   \frac{E_1}{B_1(Q_{{2}})}\norm{\B{f}^{in}}_{L^2(Q_2)} + \frac{E_1}{B_1(Q_{{1}})}\norm{\B{f}^{out}}_{L^2(Q_1)} + E_1\norm{\bu^{out}}_{Q_1}
		\label{bound_lemma_2}
		\end{equation} \end{proof}
The construction of the solution for $z$ with large magnitude will be carried out {using} the  following steps.
\begin{enumerate}
	\item $O(w^{-1})$:  Consider the system of  \eqref{order_m1_pde_Q2} and \eqref{order_m1_interface_cond} for    $\bu_0^{in}(\by;\be)\in H(Q_2)$. 
	\begin{equation}
	\left \{\begin{split}
	\text{div}_{\by} \left( 2\mu_1 e(\bu_0^{in}) \right)  &= \bo\quad \text{in }  Q_2\\
	2\mu_1e(\bu_0^{in}) \cdot \bn &= C(\by) \bn \quad \text{on } \Gamma
	\end{split}	\right.
	\label{w_m1_eqn}
	\end{equation}
		The variational formulation is
		\begin{equation}
		-\int_{\Gamma} \bar{\bv}\cdot   2\mu_1e(\bu_0^{in}) \cdot\bn   - \int_{Q_2}2\mu_1 e(\bu_0^{in}) : \overline{e(\bv)} = 0 \quad \forall \bv\in H(Q_2)
		\end{equation}
		The first term vanishes because of the boundary conditions. Hence 
	        \[
		\bu_0^{in}(\by;\be) = \bo \mbox{ in } Q_2 \mbox{ {because $H(Q_2)\perp \mathcal{R}(Q_2)$.}}
		\]
	\item $O(w^0)$ in $Q_1$: Solve the system of  \eqref{interface_cond_3} and \eqref{order_0_Q1} for $\bu_0^{out}(\by;\be) \in H(Q_1)$.
	\begin{equation}
	\label{variational_form_noslip}
	\left \{\begin{split}
	\text{div}_{\by} \left( 2\mu_1 e(\bu_0^{out}) - p_0^{out}\bI \right)  &= -\be \quad \text{in } Q_1\\
	\bu_0^{out} = \bu_0^{in} &= \bo \quad \text{on } \Gamma
	\end{split}\right.
	\end{equation}
	 An application of Lemma \ref{lemma_1}  and \eqref{u_out_bound_nonzero} leads to the following result
		\begin{equation}
		\norm{\bu_0^{out}}_{Q_1} \leq \frac{\sqrt{|Q_1|}}{B_1(Q_{{1}})} + 2E_1 \norm{\bu^{in}_0}_{Q_2}= \frac{\sqrt{|Q_1|}}{B_1(Q_{{1}})}
		\end{equation}
	\item $O(w^0)$ in $Q_2$: Consider the system of  \eqref{order_0_Q2} and \eqref{order_k_interface_cond} for $\bu_1^{in}\in H(Q_2)$:
	\begin{equation*}
	\left \{\begin{split}
	&\text{div}_{\by} \left( 2\mu_1 e(\bu_1^{in}) - p_0^{in}\bI \right)  = -\be  \quad \text{in } Q_2\\
	&\bn\times\bn\times\left[\left( \left( 2\mu_1 e(\bu_0^{out}) - p_0^{out}\bI \right) - \left( 2\mu_1 e(\bu_{1}^{in}) - p_0^{in}\bI \right)  \right)\cdot \bn\right] = \bo \quad \text{on } \Gamma
	\end{split}\right.
	\end{equation*}
By applying Lemma \ref{lemma_2} and \eqref{bound_lemma_2} with $\B{f}^{out} = -\be$ and $\B{f}^{in} = -\be$, we obtain
				\begin{equation}
			\norm{\bu_{1}^{in}}_{Q_2} \leq {C_1E_1},\quad {C_1}:=\frac{\sqrt{|Q_2|}}{B_1(Q_{{2}})} + 2\frac{\sqrt{|Q_1|}}{B_1(Q_{{1}})}
			\label{bound_on_u1_in}
		\end{equation}
	
	\item Induction step, $k\geq 1$: Given $\bu_{k}^{in} \in H(Q_2)$ and $\bu_{k-1}^{out}\in H(Q_1)$, find $\bu_{k+1}^{in}(\by;\bo) \in H(Q_2)$ and $\bu_{k}^{out}(\by;\bo)\in H(Q_1)$.
	\begin{enumerate}
		\item Applying Lemma \ref{lemma_1} with $\B{f} = \bo$, we conclude that for a given $\bu_{k}^{in} \in H(Q_2)$, $k\ge 1$, there exists a unique $\bu_{k}^{out}\in H(Q_1)$ that solves the system of \eqref{order_k_Q1} and \eqref{interface_cond_3} and assumes the estimate 
		     \begin{equation}
			\left \{\begin{split}
			\text{div}_{\by} \left( 2\mu_1 e(\bu_{k}^{out}) - p_{k}^{out}\bI \right)  &= \bo  \quad \text{in } Q_1\\
			\bu_{k}^{out} &= \bu_{k}^{in} \quad \text{on } \Gamma,
			\end{split}\right.
			\label{induction_in_to_out}
		\end{equation}
					\begin{equation}
				\norm{\bu_{k}^{out}}_{Q_1} \leq 2E_1\norm{\bu_{k}^{in}}_{Q_2}
				\label{ineq_in_to_out_induction}
			\end{equation}

		\item By applying Lemma \ref{lemma_2} with $\B{f}^{in} = \bo=\B{f}^{out}$, we see that for any given $\bu_{k}^{out}\in H(Q_1)$, $k\ge 1$ that satisfies \eqref{induction_in_to_out}, there exists a unique solution $\bu_{k+1}^{in} \in H(Q_2)$ to the system of equations \eqref{order_k_Q2} and \eqref{order_k_interface_cond}
		\begin{equation}
			\left \{\begin{split}
			\text{div}_{\by} \left( 2\mu_1 e(\bu_{k+1}^{in}) - p_{k}^{in}\bI \right)  &= \bo  \quad \text{in } Q_2\\
			\bn\times\bn\times\left[\left( \left( 2\mu_1 e(\bu_k^{out}) - p_k^{out}\bI \right) - \left( 2\mu_1 e(\bu_{k+1}^{in}) - p_k^{in}\bI \right)  \right)\bn\right] &= \bo \quad \text{on } \Gamma
			\end{split}\right.
		\end{equation}
		
Moreover, 			\begin{equation}
				\norm{\bu_{k+1}^{in}}_{Q_2} \leq E_1\norm{\bu_{k}^{out}}_{Q_1}
				\label{ineq_out_to_in_induction}
			\end{equation}
	\end{enumerate}

\end{enumerate}
Now we have found the coefficients $\bu^{in}_n(\by;\be)$ and $\bu^{out}_n(\by;\be)$  in \eqref{u_out} iteratively. We prove the convergence of the series in the following theorem by taking into account the fact that $\bu_0^{in} = \bo$.

\begin{theorem}
	Define the partial sums
	\begin{equation*}
	\B{S}_q^{in}(\by;\be,w) := \sum_{k = 0}^{q} \bu^{in}_{k+1}(\by,\be) w^{k+1},\,\B{S}_q^{out}(\by;\be,w) := \sum_{k = 0}^{q} \bu^{out}_{k}(\by;\be) w^k.
	\end{equation*}
 {Let $R\in(0,1)$, } in the disk $ |w|\le \frac{R}{2E_1^2}$, the series $\B{S}_q^{in}(\by;\be,w)$ and $\B{S}_q^{out}(\by;\be,w)$ converge uniformly to  $\bu^{in}_{\infty}(\by;\be,w)\in H(Q_2)$ and $\bu^{out}_{\infty}(\by;\be,w)\in H(Q_1)$, respectively. Therefore, $\bu_{\infty}(\by;\be,w) \defeq \bu^{in}_{\infty}(\by;\be,w)\chi_{2} + \bu^{out}_{\infty}(\by;\be,w)\chi_{1} \in H(Q)$ solves the cell problem \eqref{cell_prob} and is analytic for $ |w|< \frac{1}{2E_1^2}$.
	\label{theorem_1}
\end{theorem}
\begin{proof}
\qquad For each $q\in \mathbb{N}$, $\B{S}_{q}^{in}(\by;\be,w)$ is a polynomial function of $w$ and maps from $\mathbb{C}$ to the Hilbert space $H(Q_2)$. 
Similarly, $\B{S}_q^{out}(\by;\be,w)$ maps from $\mathbb{C}$ to $H(Q_1)$. 
To show uniform convergence, we note that \eqref{ineq_in_to_out_induction} and \eqref{ineq_out_to_in_induction} imply there exist{s a} positive constant $E_1$ that depends only on $Q_1$ and $Q_2$ such that $\norm{\bu^{in}_{k+1}}_{Q_2} \leq E_1\norm{\bu^{out}_{k}}_{Q_1} \leq 2E_1^2 \norm{\bu^{in}_{k}}_{Q_2}$. Therefore,
	\begin{equation}
\norm{\bu_k^{in}}_{Q_2} \leq \left(2E_1^2 \right)^{k-1}\norm{\bu_1^{in}}_{Q_2}, k\ge 1
		\label{MYO_in_induction}
	\end{equation}
	 Let $m>q>N$, and define $r := 2E_1^2\left| w \right|$. Then {by \eqref{bound_on_u1_in}} implies
	 \begin{align*}
	&\norm{\B{S}_m^{in}(w) - \B{S}_q^{in}(w)}_{Q_2}\leq \norm{\bu_1^{in}}_{Q_2} \left( (2E_1^2)^q |w|^{q+1} + \cdots + (2E_1^2)^{m-1}\left|  w\right|^{m} \right)\\
	&\leq \frac{r^{q+1} - r^{m+1}}{1-r} \frac{\norm{\bu_1^{in}}_{Q_2}}{2E_1^2}
	\leq \frac{r^{q+1} - r^{m+1}}{1-r}\left( {\frac{C_1}{2E_1^2}} \right) \mbox{  {,\, $C_1$ is defined in \eqref{bound_on_u1_in}}}.
	\end{align*}
Therefore, for $r\le R<1$, i.e. $|w|\le \frac{R}{2E_1^2}$, where $R$ is any fixed number in $(0,1)$,
	\begin{equation}
		\norm{\B{S}_m^{in}(w) - \B{S}_q^{in}(w)}_{Q_2} \le \left( {\frac{C_1}{2E_1^2}} \right)\left(\frac{R^{N+1}}{1-R}\right), \,\,\forall m>q>N.
			\end{equation}
For $\B{S}_q^{out}(\by;\be,w)$ we have $\norm{\bu_k^{out}}_{Q_1} \leq \left(2E_1^2 \right)^{k -1}\norm{\bu_1^{out}}_{Q_1}$. By a similar procedure,  for $m>q>N$ and $|w|\le \frac{R}{2E_1^2}$ the following estimate is valid
	\begin{equation*}
	    \norm{\B{S}_m^{out}(w) - \B{S}_q^{out}(w)}_{Q_1} \leq {C_1}\left(\frac{R^{N+1}}{1-R} \right)
	\end{equation*}
Therefore, for every fixed $w$ {satisfying} $|w|\le \frac{R}{2E_1^2}$ for any $0<R<1$, $\B{S}_q^{in}(\by;w)$ and $\B{S}_q^{out}(\by;w)$ converge uniformly to $\bu^{in}_{\infty}(\by;w)\in H(Q_2)$ and $\bu^{out}_{\infty}(\by;w)\in H(Q_1)$, respectively. Since for each $q$, $\B{S}_q^{in}(\by;w)$ and  $\B{S}_q^{out}(\by;w)$ are polynomials of $w$, hence analytic, the uniform convergence implies that the limit functions $\bu^{in}_{\infty}(\by;w)$ and $\bu^{out}_{\infty}(\by;w)$ are also analytic in $|w|<  \frac{1}{2E_1^2}$ with values in $H(Q_1)$ and $H(Q_2)$, respectively, by applying Morera's theorem for Banach space valued analytic functions \cite{1987} to the uniformly converging sequences. By construction, the function $ \bu_{\infty}(\by;\be,w)$ defined in \eqref{large_ansaz} solves the cell problem \eqref{cell_prob} for all $w$ in the disk $\{w: |w|<\frac{1}{2E_1^2}\}:=B_0(\frac{1}{2E_1^2})$.  Moreover, the uniqueness of the solution implies that  $\bu_{\infty}(\by;\be_k,w) = \bu^{k}(\by;\frac{1}{w})  \text{ in } H(Q)$ for $w\in B_0(\frac{1}{2E_1^2})\cap   \{w\in \mathbb{C}\setminus (-\infty,0]\}$. \end{proof}

The following theorem shows the relation between the two-fluid self-permeability $\bK$ in \eqref{permeability_def}  and the Darcy permeability $\bK^{(D)}$ in \eqref{K_tartar_form}
\begin{theorem}
	In the case of large viscosity $|z|>2E_1^2$ (or $|w|<\frac{1}{2E_1^2}$),  we have
	\begin{enumerate}
		\item $\bu_{\infty}^{in}(\by;\be_i,0) = \bo$ in $Q_2$
		\item As $w\goto 0$, the solution $\bu_{\infty}^{out}(\by;\be_i,w)$ converges uniformly in $\mathring{H}(Q_1)$ to the solution $\bu_D^{i}(\by)$ of the classical cell problem \eqref{classical_cell_prob}.
\item For $w\in B_0( \frac{1}{2E_1^2})$, the difference between the self-permeability $\B{K}(\by;\be_i,w)$ and the classical permeability tensor $\B{K}^{(D)}(\by;\be_i)$ satisfies $\lvert K_{ij} - (K^{(D)})_{ij} \rvert = O( |w|)$, hence $\B{K} \to \B{K}^{(D)}$ uniformly as $|w| \to 0$.
	\end{enumerate} 	
\end{theorem}
\begin{proof}
\quad The uniform convergence allows passing the limit $w \rightarrow 0$ inside the summation of  \eqref{u_out} to obtain  $\bu_{\infty}^{in}(\by;\be_i,0) = \bo$. Similarly, the uniform convergence allows passing the limit $w \rightarrow 0$ inside the summation of  \eqref{u_out} to obtain
		\[
		\bu^{out}_{\infty}(\by;\be_{{i}},0) =\bu_0^{out}(\by;\be_i)
		\]		
		Furthermore, $\bu_0^{out}(\by;\be_i) \in H(Q_1)$ satisfies \eqref{variational_form_noslip} and in fact $\bu_0^{out}(\by;\be_i) \in \mathring{H}(Q_1)$ since $\bu_0^{out}(\by;\be_i)\vert_{\Gamma} = \bo$, which is identical to the equation for $\bu_D$ \eqref{classical_cell_prob}. The uniqueness of the solution then ensures that $\bu_0^{out}(\by;\be_i) = \bu_D^{i}(\by)$. Therefore the series $\bu_{\infty}^{out}(\by;\be_i,w) \to \bu_D^{i}(\by)$ uniformly as $|w|\to 0$ in $\mathring{H}(Q_1)$. For (iii), we note that 
		\begin{equation*}
		\begin{split}
		\left|  K_{ij} (w)- K^{(D)}_{ij} \right|  &= \left|  \int_{Q} \left( \bu^i - \chi_1 \bu_D^i\right) \cdot \be_j d\by\right| \leq \norm{\bu^i - \chi_1\bu^i_D}_{L^2(Q)}\\
		&\le \frac{1}{B_1(Q)}  \norm{\sum_{k=1}^{\infty} \left( \bu_{k}^{in}(\by;\be_i)\chi_{2} + \bu_{k}^{out}(\by;\be_i)\chi_{1} \right)w^k }_Q 
		\end{split}
		\end{equation*}
		From \eqref{ineq_in_to_out_induction}, \eqref{MYO_in_induction} and \eqref{bound_on_u1_in}, we have for $|w|<\frac{1}{2E_1^2}$, or equivalently $|z| > 2E_1^2$, 
		\begin{equation}
		\label{large_z_asymp}
		\left|  K_{ij} (w)- K^{(D)}_{ij} \right| \leq 
		{C_1\left( \frac{E_1+1}{2E_1 B_1(Q)}\right)}\frac{2E_1^2 |w|}{1-2E_1^2 |w|}	
		\end{equation}\end{proof}
In the following section, we study the behavior of $\bK(z)$ near $z=0$, i.e. the inclusion is an air bubble. 
\subsection{Analyticity of the solution for small $|z|$}
Let $\be$ be a constant unit vector in $\mathbb{R}^n$. We seek solutions of the following form 
\begin{align}
\bu_{null}^{in}(\by;\be,z) &= \sum_{k = 0}^{\infty} \bu^{in}_k(\by;\be) z^k ,\quad  p^{in}(\by;\be,z) = \sum_{k = 0}^{\infty} p^{in}_k(\by;\be) z^k \text{ in } Q_2, \label{u_in_prime}\\
\bu_{null}^{out}(\by;\be,z) &= \sum_{k = 0}^{\infty} \bu^{out}_k(\by;\be) z^k,\quad  p^{out}(\by;\be,z) = \sum_{k = 0}^{\infty} p^{out}_k(\by;\be) z^k \text{ in } Q_1\label{u_out_prime}
\end{align}
By a procedure similar to that in Section \ref{large_z}, the following equations are obtained via collecting terms with respect to the order of $z$. The PDEs for $Q_1$ are as follows.
\beqa
	O(1)&:&\qquad \text{div}_{\by} \left( 2\mu_1 e(\bu_0^{out}) - p_0^{out}\bI \right)  = -\be \label{order_0_Q1_prime}\\	
        O(z^k), \,k\geq 1&:&\qquad \text{div}_{\by} \left( 2\mu_1 e(\bu_k^{out}) - p_k^{out}\bI \right)  = \bo
	\label{order_k_Q1_prime}
\eeqa	
Similarly, the PDEs for $Q_2$ are
\beqa
	O(1)&:&\qquad -\nabla p_0^{in}=-\be \label{p0_prime}	\\
	O(z^{k}),\,k\geq 1 &:&\text{div}_{\by} \left( 2\mu_1 e(\bu_{k-1}^{in}) - p_k^{in}\bI \right)  = \bo \label{order_k_Q2_prime}
 \eeqa
The interface condition \eqref{interface_cond_3} remains the same for the small $|z|$ case while \eqref{order_m1_interface_cond} and \eqref{order_k_interface_cond}  now read
	\begin{eqnarray}
	&&    \bn \times \bn \times \left[ \left( \left( 2\mu_1 e(\bu_0^{out}) - p_0^{out}\bI \right)\cdot\bn - ( - p_0^{in}\bI   \right) \cdot\bn \right] = \bo \label{interface_prime_0}\\
	   && \bn \times \bn \times \left[ \left( \left( 2\mu_1 e(\bu_k^{out}) - p_k^{out}\bI \right) 
	    - \left( 2\mu_1 e(\bu_{k-1}^{in})- p_k^{in}\bI   \right)\right) \cdot\bn \right]= \bo,\, k\ge 1
	\end{eqnarray}
The first equation to be solved is \eqref{p0_prime}, whose solution is simply
\begin{equation}
p_0^{in}(\by)=\be\cdot \by+{c}-\int_{Q_2} (\be\cdot \by+{c}) d\by \mbox{ in } Q_2
\end{equation}
where ${c}$ is a constant. The next problem is the system of \eqref{order_0_Q1_prime} and \eqref{interface_prime_0}. Similar to the calculation in Lemma \ref{lemma_2}, the weak formulation of this system is: Find $u_0^{out} \in H(Q_1)$ such that for all $\bphi\in H(Q_1)$ and $\pi_0^{out}:=2\mu_1 e(\bu_0^{out}) - p_0^{out}\bI$
\beqs
 - \int_{\Gamma} (\bar{\Phi} \cdot ( \pi_0^{out}+p_0^{in}\bI)-p_0^{in}\bI)\cdot \bn dS - \int_{Q_1} 2\mu_1 e(\bu_0^{out}) : \overline{e(\Phi)} d\by = \int_{Q_1} -\be \cdot \bar{\Phi} d\by
\eeqs
Since $\bphi\cdot \bn=0$ and $p_0^{in} \bI \cdot \bn$ is parallel to $\bn$, \eqref{interface_prime_0} implies the integral on $\Gamma$ vanishes. Hence by {the} Lax-Milgram lemma, we have 
\beq
\|\bu_0^{out}\|_{Q_1}\le \frac{\sqrt{|Q_1|}}{B_1(Q_{{1}})}
\eeq
The system for $\bu_{k-1}^{in}$, $k\ge 1$ (inner problem) is to find $\bu_{k-1}^{in}\in H(Q_2)$ with given $\bu_0^{out}\in H(Q_1)$ such that
\beq
\label{bubble_in_prob}
\left\{
\begin{split}
\text{div} \left( 2\mu_1 e(\bu_{k-1}^{in}) - p_k^{in}\bI \right)  = \bo \mbox{ in } Q_2\\
\bu_{k-1}^{in}|_\Gamma=	\bu_{k-1}^{out}|_\Gamma    
\end{split}
\right.
\eeq
With an argument similar to the derivation of Lemma \ref{lemma_1}, the following estimate can be derived for system \eqref{bubble_in_prob}
\begin{lemma}
	Let $Q_2$ satisfy the same assumption in Lemma \ref{lemma_1}. For any given vector field $\bu^{out} \in H(Q_1)$, there exists a unique weak solution $\bu^{in}(\by)\in H(Q_2)$ s.t.
	\beqa
	\left\lbrace
	\begin{split}
	    \text{div}_{\by} \left( 2\mu_1 e(\bu^{in}) - p^{in}\bI \right) &= \mathbf{f}^{in} \quad \text{ in } Q_2\\
	    \bu^{in} &= \bu^{out} \quad \text{ on }\Gamma
	\end{split}
	\right.
	\label{lemma_1_eqn_prime}
\\
 \norm{\bu^{in}}_{Q_2} \leq \frac{1}{B_1(Q_{{2}})} \norm{\mathbf{f}^{in}}_{L^2(Q_2)} + 2E_2\norm{\bu^{out}}_{Q_1}.	
 \label{bubble_in}
	\eeqa
	where $E_2>1$ is the constant {associated with the extension operator} $T$, $\|T(\bphi)\|_Q\le E_2 \| \bphi\|_{Q_{{1}}}$ for all $\bphi\in H(Q_{{1}})$ and $T(\bphi)$ decays rapidly to 0  inside $Q_2$. {Note that the periodic condition of space $H(Q_1)$ implies  $\int_\Gamma \bu^{out} \cdot \bn \,dS=0$.}
	\label{lemma_1_prime} 
\end{lemma}
The system for $\bu_k^{out}$ and $p_k^{out}$  with given $\bu_{k-1}^{in}\in H(Q_2)$ and $p_k^{in}$, $k\ge1$ is
\beq
\label{bubble_out_prob}
\left\{
\begin{split}
\text{div}\left( 2\mu_1 e(\bu_k^{out}) - p_k^{out}\bI \right)  = \bo\\
	    \bn \times \bn \times \left[ \left( \left( 2\mu_1 e(\bu_k^{out}) - p_k^{out}\bI \right) 
	    - \left( 2\mu_1 e(\bu_{k-1}^{in})- p_k^{in}\bI   \right)\right) \cdot\bn \right]= \bo
\end{split}
\right.
\eeq
By an argument similar to the one for Lemma \ref{lemma_2}, the system above can be shown to satisfy the following estimate.
\begin{lemma}
	Let $Q_2$ satisfy the same assumption in Lemma \ref{lemma_1}. For any given pair of $\left( \bu^{in}, p^{in}\right) \in H(Q_2) \times L^2(Q_1)/ \mathbb{C}$ that satisfies \eqref{lemma_1_eqn_prime}, there exists a unique vector $\bu^{out}(\by;\B{f}^{out})\in H(Q_1)$ {solving} the following system
\beqa
	\left\lbrace
	\begin{split}
	&   \text{div} \left( 2\mu_1 e(\bu^{out}) - p^{out}\bI \right)  = \B{f}^{out} \text{ in }Q_1\\
	 &   \bn\times\bn\times\left[\left( \left( 2\mu_1 e(\bu^{in}) - p^{in}\bI \right) - \left( 2\mu_1 e(\bu^{out}) - p^{out}\bI \right)  \right)\cdot\bn\right] = \bo \text{ on }\Gamma
	\end{split}
	\right.
	\label{lemma_2_eqn_prime}
\\
		\norm{\bu^{out}}_{Q_1} \leq   \frac{E_2}{B_1(Q_{{1}})}\norm{\B{f}^{out}}_{L^2(Q_1)} + \frac{E_2}{B_1(Q_{{2}})}\norm{\B{f}^{in}}_{L^2(Q_2)} + E_2\norm{\bu^{in}}_{Q_2}
\label{bubble_out}
\eeqa
where $E_2$, $B_1$ depend only on $Q$ and $\Gamma$.		
	\label{lemma_2_prime} 
	\end{lemma}
Equation \eqref{bubble_out_prob}, Lemma \ref{lemma_1_prime}, Equation \eqref{bubble_in_prob} and Lemma \ref{lemma_2_prime} imply that for all $k\ge 0$, we have $\|\bu_k^{in}\|_{Q_2}\le 2E_2 \|\bu_k^{out}\|_{Q_1}$ and $\|\bu_{k+1}^{out}\|_{Q_1}\le E_2 \|\bu_k^{in}\|_{Q_2}$. Therefore,
\beqa
\| \bu^{in}_k\|_{Q_2} \le 
\frac{(2E_2^2)^{k+1}}{E_2} \|\bu_0^{out}\|_{Q_1}\le {(2E_2^2)^{k+1}}\left( \frac{|Q_1|}{E_2B_1(Q_{{1}})}\right)\\
\| \bu^{out}_k\|_{Q_1} \le 
 (2E_2^2)^k  \|\bu_0^{out}\|_{Q_1}\le (2E_2^2)^k \left( \frac{|Q_1|}{B_1(Q_1)}\right)
\eeqa
Therefore, the series in  \eqref{u_in_prime} and \eqref{u_out_prime} converge uniformly in the dis{k} $|z|<\frac{1}{2(E_2)^2}$ to an analytic function in $Q_2$ and $Q_1$, respectively. The limit functions  $\bu_{null}^{in}(\by,\be,z)$, $\bu_{null}^{out}(\by,\be,z)$  and  the corresponding permeability $K_{ij}(z)$ in \eqref{permeability_second_form} are analytic at $z=0$. Define the permeability ('B' for 'bubbles)
\beq
K_{ij}^{(B)}:=\int_Q [\chi_1 \bu_0^{out}(\by;\be_i) + \chi_2  \bu_0^{in}(\by;\be_i)]\cdot \be_j\, d\by
\eeq 
then the following estimate{,} valid for $|z|<\frac{1}{2E_2^2}${,} holds
\beqa
|K_{ij}(z)-K^B_{ij}| 
\le \frac{|Q_1|(1+2E_2)}{{B_1(Q) B_1(Q_1)}}\left(\frac{2E_2^2|z|}{1-2E_2^2|z|} \right) =O(|z|). \label{small_z_asymp}
\eeqa    
In conclusion,  $\bK(z)$ in \eqref{permeability_def}  is analytic for $z\in \mathbb{C}\setminus[-2E_1^2,-\frac{1}{2E_2^2}]$, $E_1, E_2\ge 1$. In the next section, and integral representation formula (IRF) for $\bK(z)$ will be derived in two different ways.
\section{Integral representation of permeability $\bK(z)$}
\label{Section3}
We first observe two properties of $\bK$ implied by  \eqref{permeability_second_form}.
\begin{proposition}
\beqa
\frac{\bK(z)-\bK^*(z)}{z-\bar{z}}\le 0 \mbox{ if } Im(z)\ne 0 \label{prop_1}\\
\bK(x)\ge 0 \mbox{ for }x>0  \label{prop_2}
\eeqa
\label{prop_K_z}
\end{proposition}
\begin{proof} \quad Note that $K_{ij}(z)-(K^*)_{ij}(z)=  {2\mu_1}(\overline{z}-z)\int_{Q_2}  {e( \bu^j(z))} : \overline{e(\bu^i(z))}  \,d\by$. Hence
\[
\frac{K_{ij}(z)-K^*_{ij}(z)}{z-\overline{z}}= {-}{2\mu_1} \int_{Q_2}  {e(\bu^j(z))} : \overline{e( \bu^i(z))}\, d\by= {-} (\bu^j,\bu^i)_{Q_2}=:-A_{ij}
\]
The matrix $\pmb{A}$ is obviously Hermitian and for any $\pmb{\xi}\in \mathbb{C}^n$, we have $\overline{\xi_i} A_{ij} {\xi_j}= (\xi_j\bu^j,\xi_i \bu^i)_{Q_2}\ge 0$. This proves \eqref{prop_1}.  Recall that $K_{ij}(x)= \left(  (\bu^j,\bu^i)_{Q_1}+x  (\bu^j,\bu^i)_{Q_2} \right)$. With a similar argument, \eqref{prop_2} follows.
\end{proof}
With these two properties and the fact that $\bK$ is holomorphic in $\mathbb{C}\setminus(-\infty,0]$,  the characterization theorem for matrix-valued function{s belonging to the} Stieltjes class \cite{dyukarev1986multiplicative} implies that there exists a monotonically increasing matrix-valued function $\pmb{\sigma}(t)$ such that the following integral representation formula (IRF) holds for $z\in \mathbb{C}\setminus(-\infty,0]$
\[
\bK(z)=\pmb{A}+\frac{\pmb{C}}{z}+\int_{+0}^\infty \frac{1}{z+t} d\pmb{\sigma}(t)
\]
where $\pmb{A}\ge 0$, $\pmb{C}\ge 0$, $\int_{+0}^\infty \frac{1}{1+t} d\pmb{\sigma}(t){:=\displaystyle{\lim_{\epsilon\downarrow 0}\int_{\epsilon}^\infty }\frac{1}{1+t} d\pmb{\sigma}(t)}<\infty$ and $\pmb{A}+\pmb{C}+\int_{+0}^\infty \frac{1}{1+t} d\pmb{\sigma}(t)>0$. Since $\bK(0)=\bK^{(B)}$, we must have $\pmb{C}=\pmb{0}$. Also, $\pmb{K}(\infty)=\pmb{K}^{(D)}$ implies $\pmb{A}=\pmb{K}^{(D)}$
\[
\bK(z)=\pmb{K}^{(D)}+\int_{\frac{1}{2E_2^2}}^{2E_1^2}\frac{1}{z+t} d\pmb{\sigma}(t)
\]
Therefore, for real valued $z$, $\bK(z)$ is decreasing as $z$ increases, {i.e., $\bK(x_1)-\bK(x_2)$ is negative semidefinite if $x_1>x_2$}. To study how the measure $d\pmb{\sigma}$ is related to the microstructure, we derive the spectral representation of $\bK(z)$ by using the underlying system \eqref{cell_prob}.%
\subsection{Spectral representation of $\bK(z)$}
Adding $\int_{Q_2}2\mu_1 e( \bu^k) : \overline{e(\bv)}  d\by$ to both sides of \eqref{variational_form_cell_prob}, we have
\begin{equation}
    \int_{Q}2\mu_1 e( \bu^k) : \overline{e(\bv)}  d\by = -\frac{1}{s}\int_{Q}2\mu_1  \chi_2 e( \bu^k) : \overline{e(\bv)}  d\by + \int_{Q} \be_k\cdot \bar{\bv}  d\by
    \label{equate_LHS_RHS_spectral_rep}
\end{equation}
where the new variable $s$ is defined as
\[
s:=\frac{1}{z-1}
\]
Let  $\Delta_{\#}^{-1}$ be the operator that solves for $\B{w}(\by;\B{f}) \in H(Q)$ in the following variational formulation
\begin{equation}
    \int_{Q} 2\mu_1 e(\bw):\overline{e(\bv)} d\by = \int_{Q}\B{f}\cdot \bar{\bv} d\by
\end{equation}
where $\B{f} \in L^2(Q)$ and $Q$-periodic. In other words, solution $\bw(\by) = \Delta_{\#}^{-1} \B{f} \in H(Q)$ is a weak solution to the cell problem
\begin{equation}
\left \{ \begin{split} -\mu_1\Delta \bw  &= \B{f} \quad \text{in } Q_1 \cup Q_2\\
\llbracket \bpi \rrbracket\bn  &= \left( \llbracket \bpi \bn\rrbracket \cdot \bn \right) \bn \text{ on } \Gamma
\end{split}\right.
\end{equation}
In order to get the spectral representation, we apply $\Delta_{\#}^{-1}$ on both sides of \eqref{equate_LHS_RHS_spectral_rep} and symbolically represent the resulted equations as 
\[
 \bw_1 = {-\frac{1}{s}}\bw_2 + \bw_3
\]
Then clearly, we have $ \bw_1 = \bu^k $ and $  \bw_3 =\Delta_{\#}^{-1}\be_k$. Observe that $\bw_2$ solves
    \begin{equation}
        \int_{Q} 2\mu_1e(\bw_2):\overline{e(\bv)} d\by = \int_{Q} 2\mu_1 \chi_2 e( \bu^k) : \overline{e(\bv)}  d\by  \mbox{ for all } \bv\in H(Q)
    \label{pre-definition_of_gamma_chi}
    \end{equation}
Define the operator $\Gamma_{\chi}$ such that $\bw_2=\Gamma_{\chi} \bu^k$ and \eqref{pre-definition_of_gamma_chi} can be expressed as

    \begin{equation}
      (\Gamma_\chi \bu^{{k}},\bv)_Q = \int_{Q} 2\mu_1 \chi_2 e( \bu^{{k}}) : \overline{e(\bv)}  d\by \mbox{ for all } \bv\in H(Q).
        \label{variational_form_gamma_chi}
    \end{equation}
    {The subscript $\chi$ is used to signify the dependence of $\Gamma_\chi$ on $\chi_2$, the characteristic function of $Q_2$.} Clearly, $\Gamma_\chi$ is self-adjoint with respect to the inner product $(\cdot,\cdot)_Q$ because
    \[
   (\Gamma_\chi \bu,\bv)_Q=\overline{\int_{Q}  2\mu_1\chi_2 e( \bv) : \overline{e(\bu)}  d\by}=  \overline{(\Gamma_\chi \bv,\bu)_Q } =  {(\bu,\Gamma_\chi \bv)_Q. }
    \]
Formally, we have $\Gamma_\chi \bu=\triangle_\#^{-1}(\nabla\cdot \chi_2 e(\bu))$.  Now \eqref{equate_LHS_RHS_spectral_rep} becomes 
\begin{equation}
    \bu^k = -\frac{1}{s} \Gamma_{\chi} \bu^k + \Delta_{\#}^{-1}\be_k
    \Leftrightarrow  \left( I +\frac{\Gamma_{\chi}}{s} \right)\bu^k = \Delta_{\#}^{-1}\be_k
\label{needed_after_self_adjoint_lemma}
\end{equation}
\begin{proposition}
The self-adjoint operator $\Gamma_{\chi}$ defined in \eqref{variational_form_gamma_chi} is positive and bounded with $\norm{\Gamma_{\chi}}\le 1$.
	\label{lemma_self_adj_gamma_chi}
\end{proposition}
\begin{proof}
 \quad  It can be proved by choosing $\bv=\bu$ in \eqref{variational_form_gamma_chi} and observe that 
  $
  0\le \int_{Q_2} 2\mu_1e(\bu):\overline{e(\bu)} d\by \le \int_{Q} 2\mu_1e(\bu):\overline{e(\bu)} d\by= (\bu,\bu)_Q
   $.
\end{proof}
\begin{theorem}
	For  $|s|>1$, the solution $\bu^k\in H(Q)$ admits a series representation
	\begin{equation}
	\bu^k(\by;s) =  \sum_{m = 0}^{\infty}\left( -\frac{1}{s} \right)^{m}\left(\Gamma_{\chi} \right)^{m} \Delta_{\#}^{-1}\be_k 
	\label{u_k_series_explicit}
	\end{equation}
	and  the components of $\bK$ can be represented by the following IRF
	\begin{equation}
	    K_{kl}(s) = s\int_0^1\int_{Q}\frac{\left(\tilde{M}(d\lambda)\Delta_{\#}^{-1}\be_k\right)_l}{s+\lambda}\, d\by,\quad k,l=1,\dots,n,
	    \label{integral_rep_K_new}
	\end{equation}
	for some projection-valued measures $\tilde{M}(d\lambda)$ and a series representation 
	\begin{equation*}
	K_{kl}(s) =  \int_{Q} \left( \Delta_{\#}^{-1}\be_k\right)_l d\by + \sum_{m = 1}^{\infty} \frac{\tilde{\lambda}_{kl}^{m}}{(-s)^{m}}\quad \mbox{with }\tilde{\lambda}_{kl}^{m} :=   \int_{Q} \left( \left(\Gamma_{\chi} \right)^{m} \Delta_{\#}^{-1}\be_k \right)_l \,d\by.
	\end{equation*}
\end{theorem}
\begin{proof}
\quad From  \eqref{needed_after_self_adjoint_lemma}, since $\Gamma_{\chi}$ is self-adjoint with norm bounded by 1, for $|s|>1$, the spectral theory for self-adjoint operator implies the existence of a projection-valued measure $\tilde{M}$ such that 
\beq
    \bu^k(\by;s) = \left( I + \frac{\Gamma_{\chi}}{s} \right)^{-1} \Delta_{\#}^{-1}\be_k 
    = {s}\int_0^1 \frac{\tilde{M}(d\lambda)\left(\Delta_{\#}^{-1}\be_k\right)}{s+\lambda}\label{uk_solution}
\eeq
Hence the $kl$-the element of permeability $\bK$ has the following IRF
\begin{equation}
    K_{kl}(s) =  \int_{Q} (\bu^k)_l d\by=s\int_0^1\int_{Q}\frac{\left(\tilde{M}(d\lambda)\Delta_{\#}^{-1}\be_k\right)_l}{s+\lambda} d\by
    \label{K_comp_exp}
\end{equation}
On the other hand, for $|s|>1$,  the geometric expansion of the middle termnear $s=\infty$ in  \eqref{uk_solution} results in the following expression 
\begin{equation}
		K_{kl}(s) =  \int_{Q} \left[  \sum_{m = 0}^{\infty}\left(-\frac{1}{s} \right)^{m}\left( \Gamma_{\chi} \right)^{m} \Delta_{\#}^{-1}\be_k \right]\cdot \be_l d\by= \sum_{m = 0}^{\infty} \frac{\tilde{\lambda}_{kl}^{m}}{(-s)^{m}}
	\end{equation}
	where  $\tilde{\lambda}_{kl}^{m}$ is defined as 
$	    \tilde{\lambda}_{kl}^{m} :=   \int_{Q} \left( \left(\Gamma_{\chi} \right)^{m} \Delta_{\#}^{-1}\be_k \right)_l d\by.
$
	\end{proof}
For the three-dimensional space $n=3$, the expansion \eqref{K_comp_exp} can be cast in the matrix form 
\beq
\bK(s)=\sum_{m = 0}^{\infty} \frac{\tilde{\bLam}_{m}}{(-s)^{m}}
\label{Matrix-K-spectral-expan}
\eeq
with the matrix-valued moments defined as  
\beq
\tilde{\bLam}_m:= 
\begin{pmatrix} \int_{Q}  \left(\Gamma_{\chi} \right)^{m} \Delta_{\#}^{-1}\be_1 d\by &   \int_{Q} \left(\Gamma_{\chi} \right)^{m} \Delta_{\#}^{-1}\be_2 d\by &
 \int_{Q}  \left(\Gamma_{\chi} \right)^{m} \Delta_{\#}^{-1}\be_3 d\by
 \label{geo_info_terms}
 \end{pmatrix}
\eeq
\subsection{Relationships between two representations and characterization of the microstructral information on permeability}
The calculation{s} in the previous section reveal that the variable $s:=\frac{1}{z-1}$ is the natural one to use. Because of this, we will consider $\bK$ as a function of $s$. Note that $s$ maps $(-\infty,0]$ on the $z$-plane to $[-1,0]$ on the $s$-plane.The following properties of $\bK(s)$ can be easily deduced from the results in Proposition \ref{prop_K_z}
.
\begin{enumerate}
\item
$\bK(s)$ is holomorphic in $\mathbb{C}\setminus{[-\frac{2E_2^2}{1+2E_2^2},-\frac{1}{1+2E_1^2}]}.$
\item
$\frac{\bK(s)-(\bK(s))^*}{s-\overline{s}}\ge 0$ for all  $Im{(s)}\ne 0$
\item
$\bK(s)\ge 0 \mbox{ for } \mathbb{R}\ni s>0$ because $s>0$ iff $\mathbb{R}\ni z>1$.  
\end{enumerate}

Then by the representation theorem in \cite[{Theorem 3.1}]{dyukarev1986multiplicative}, there exists a monotonically increasing matrix-valued function $\pmb{\sigma}(t)$, matrices $\pmb{A}\ge 0$ and $\pmb{C}\ge 0$ such that $\int_{+0}^\infty \frac{d\pmb{\sigma}}{1+t}<\infty$, $\pmb{A}+\pmb{C}+\int_{+0}^\infty \frac{d\pmb{\sigma}}{1+t}>0$ and 
\beq
\bK(s)=\pmb{A}+{\pmb{C}}{s}+\int_{+0}^\infty \frac{s}{s+t}d\pmb{\sigma}(t), 
\eeq
As $s\goto \infty$, $z\goto 1$ and hence $\bK\goto \bK(z=1)$. Therefore, we must have $\pmb{C}=\pmb{0}$. Moreover, $\pmb{A}=\bK(s=0)=\bK^{(D)}$. Also, since $\bK(s)$ is holomorphic in $\mathbb{C}\setminus {[-\frac{2E_2^2}{1+2E_2^2},-\frac{1}{1+2E_1^2}]}$, we have
\beq
\bK(s)=\pmb{K}^{(D)}+\int_{\frac{1}{1+2E_1^2}}^{\frac{2E_2^2}{1+2E_2^2}} \frac{s}{s+t}d\pmb{\sigma}(t), 
\label{IRF_Ks_prime}
\eeq
which is valid for all $s\in \mathbb{C}\setminus{[-\frac{2E_2^2}{1+2E_2^2},-\frac{1}{1+2E_1^2}]}\subset (-1,0)$.
To compare with \eqref{Matrix-K-spectral-expan}, which is valid only for $|s|>1$,  we expand \eqref{IRF_Ks_prime} near $s=\infty$ to obtain the following series expansion
\beq
\bK(s)=\bK^{(D)}+\sum_{m=0}^\infty (-1)^m \left( \frac{1}{s} \right)^{m+1} \boldsymbol{\mu}^\sigma_m
\eeq
where $\boldsymbol{\mu}^\sigma_m$ is the $m$-th moment of the measure $d\pmb{\sigma}$. Equating the coefficients term by term with \eqref{Matrix-K-spectral-expan} leads to the following relation between  $\boldsymbol{\mu}^\sigma_m$ and the 'geometrical information' coefficients in \eqref{geo_info_terms}
\begin{eqnarray}
\bK^{(D)}+\boldsymbol{\mu}^\sigma_0=\tilde{\bLam}_0={\bK(s=\infty)},\mbox{ i.e. }&&\boldsymbol{\mu}^\sigma_0=\bK(z=1)-\bK^{(D)} \\
&&\boldsymbol{\mu}^\sigma_{m}=\tilde{\bLam}_m, \, m\ge 1
\label{moment_relation}
\end{eqnarray}
{Recall that $\bK$ can be regarded as a function of $s$ as well as a function of $z$,  $s:=\frac{1}{z-1}$. }
In particular,  the first moment $\boldsymbol{\mu}^\sigma_1$ can be calculated explicitly as follows
\beqa
 \tilde{\lambda}_{kl}^{1}=(\Gamma_{\chi}\bu^k(\by;1),\bu^l(\by;1))_Q= 2\mu_1\int_Q \chi_2 e(\bu^k(\by;1)):\overline{e(\bu^l(\by;1))} \, d\by
\eeqa
\section{Numerical verification}
\label{Section4}
\def\b#1{{\bf{#1}}}\def\an#1{\begin{align}#1 \end{align}}
\def\ad#1{\begin{aligned}#1 \end{aligned}}
\def\a#1{\begin{align*}#1 \end{align*}}\def\t#1{\text{#1}}

The computational domain with $Q=(0,1)^2$, $Q_2=[1/4,3/4]^2$ and $Q_1=Q\setminus Q_2$ is illustrated in Figure \ref{sq}.
 is chosen in the first two numerical examples, \eqref{rough} and \eqref{slippery}. 

\begin{figure}[htb] \setlength\unitlength{1pt}
  \begin{center}  
  \begin{picture}(100,100)(0,0)
    \multiput(0,0)(100,0){2}{\line(0,1){100}}
    \multiput(0,0)(0,100){2}{\line(1,0){100}}
    \multiput(25,25)(50,0){2}{\line(0,1){50}}
    \multiput(25,25)(0,50){2}{\line(1,0){50}}
    \put(3,3){$Q_1$}\put(78,28){$\tilde\Gamma$} \put(-22,3){$Q:$}
    \put(28,28){$Q_2$} \put(46,90){\vector(1,0){13}} \put(60,87){$\b e^{(1)}$}  
    \end{picture} \end{center}
\caption{ \label{sq} Computational domain }
\end{figure}
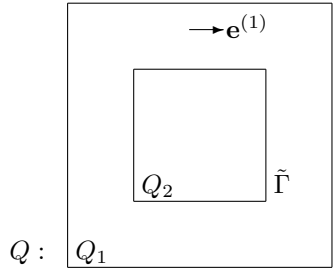

We consider three cases: (1) $Q_2$ is a solid obstacle, (2) $Q_2$ is a bubble,  
  and (3) $Q_2$ is another fluid.

For case (1), we find 
  $(\b u_1, p_1)\in \b V_1 \times P_1$, such that
\an{ \label{rough}  \left\{ \ad{
     (e(\b u_1), e(\b v)) -(p_1, \t{div} \b v) &= (\b e_1,\b v) \quad \forall \b v\in \b V_1,\\
     (q,\t{div}\b u_1 ) & = 0   \qquad \forall q \in P_1,} \right.
  } where \a{ \b V_1 &= \{\b v\in H^1(Q_1)^2 \ \mid 
                       \b v|_{\partial Q_2}=\b 0,  \t{\ $\b v$ is } Q\t{-periodic} \}, \\
           P_1 &= \{q \in L^2_0(Q_1)  \ \mid 
                         q=\t{div} \b v \t{ \ for some $\b v\in \b V_1$ } \}.  }

For case (2),  we find $(\b u_2, p_2)\in \b V_2 \times P_2$, such that
\an{ \label{slippery}  \left\{ \ad{
     (e(\b u_2), e(\b v)) -(p_2, \t{div} \b v) &= (\b e_1,\b v) \quad \forall \b v\in \b V_2,\\
     (q,\t{div}\b u_2) & = 0   \qquad \forall q \in P_2,} \right.
  } where \a{ \b V_2& = \{\b v\in H^1(Q_1)^2 \ \mid 
    \b v\cdot \b n|_{\partial Q_2}= 0, \t{\ $\b v$ is } Q-\t{periodic} \}, \\
           P_2 &= \{q \in L^2_0(Q_1)  \ \mid 
                         q=\t{div} \b v \t{ \ for some $\b v\in \b V_2$ } \}.   }

For case (3), we set $\mu_1=1$ and $\mu_2=\mu$.  We find $(\b u_3, p_3)\in \b V_3 \times P_3$, such that
\an{ \label{porous}  \left\{ \ad{
     (\mu e(\b u_3), e(\b v)) -(p_2, \t{div} \b v) &= (\b e_1,\b v) \quad \forall \b v\in \b V_3,\\
     (q,\t{div}\b u_2) & = 0   \qquad \forall q \in P_3,} \right.
  } where \a{ \b V_3& = \{\b v\in H^1(Q)^2 \ \mid 
    \b v\cdot \b n|_{\partial Q_2}= 0, \t{\ $\b v$ is } Q-\t{periodic} \}, \\
           P_3 &= \{q \in L^2_0(Q)  \ \mid 
                         q=\t{div} \b v \t{ \ for some $\b v\in \b V_3$ } \}.   }

The computation is done on square grids.  
The first level grid consists of 12 squares, for the first two cases. 
Each square is subdivided into
   4 sub-squares to get the next level grid, $\mathcal{T}_h=\{T\}$.
We use the $Q_{5,4}^{1,0}\times Q_{4,5}^{0,1}$ velocity finite element space with
   the $Q_{4,4}^{0,0}$ pressure finite element space.
  Here $Q_{5,4}^{1,0}$ means the space of polynomials of degree at most 5 in $y_1$
     and  of degree at most $4$ in $y_2$ which is 
     $C^1$ in $y_1$-direction and $C^0$ in $y_2$-direction.
That is, \a{Q_{5,4}^{1,0}&=\Big\{ u_1|_T  = \sum_{i=0}^5\sum_{j=0}^4 c_{ij} y_1^i y_2^j \ \Big| \
      u_1 \hbox{ and } \partial_{y_1} u_1 \in C^0(Q_1), \hbox{and } Q\hbox{-periodic} \Big\}, \\ 
      Q_{4,4}^{0,0}&=\Big\{ p|_T  =\sum_{i=0}^4\sum_{j=0}^4 c_{ij} y_1^i y_2^j \ \Big| \
      p \in C^0_0(Q), \hbox{and } Q\hbox{-periodic} \Big\}. }
We note that $\t{div} ( Q_{5,4}^{1,0}\times Q_{4,5}^{0,1}) = Q_{4,4}^{0,0}$.
Therefore, the finite element velocity is also pointwise divergence-free.
We plot the velocity field of these two problems in Figure \ref{2vec}.
We can see the magnitude of the latter is much bigger, as the resistance from a
   slippery bubble is much less. 

\begin{figure}[htb]\begin{center}\setlength\unitlength{1in}
    \begin{picture}(4.7,2.3)
 \put(0,0){\includegraphics[width=2.3in]{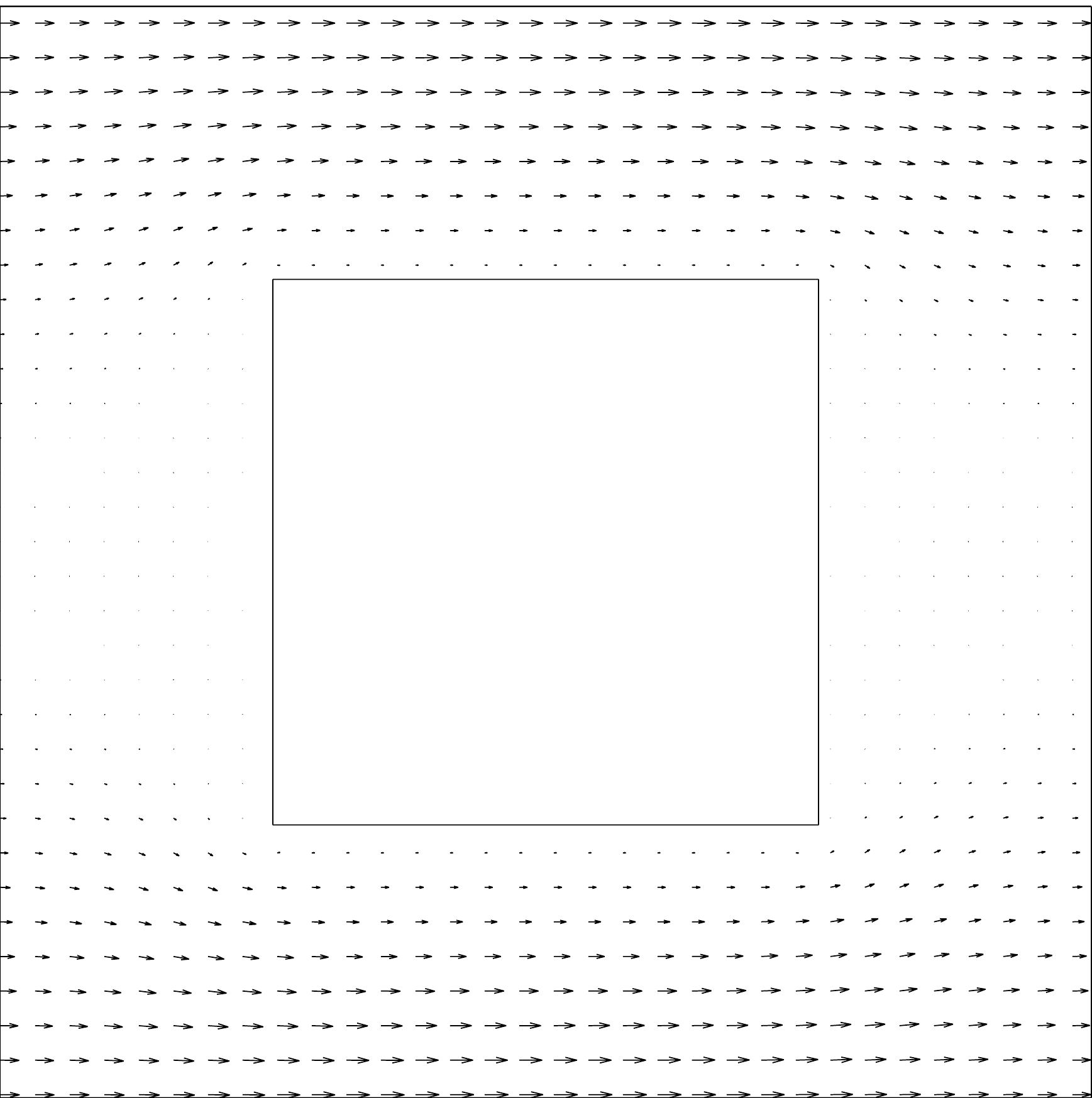}}
  \put(2.4,0){\includegraphics[width=2.3in]{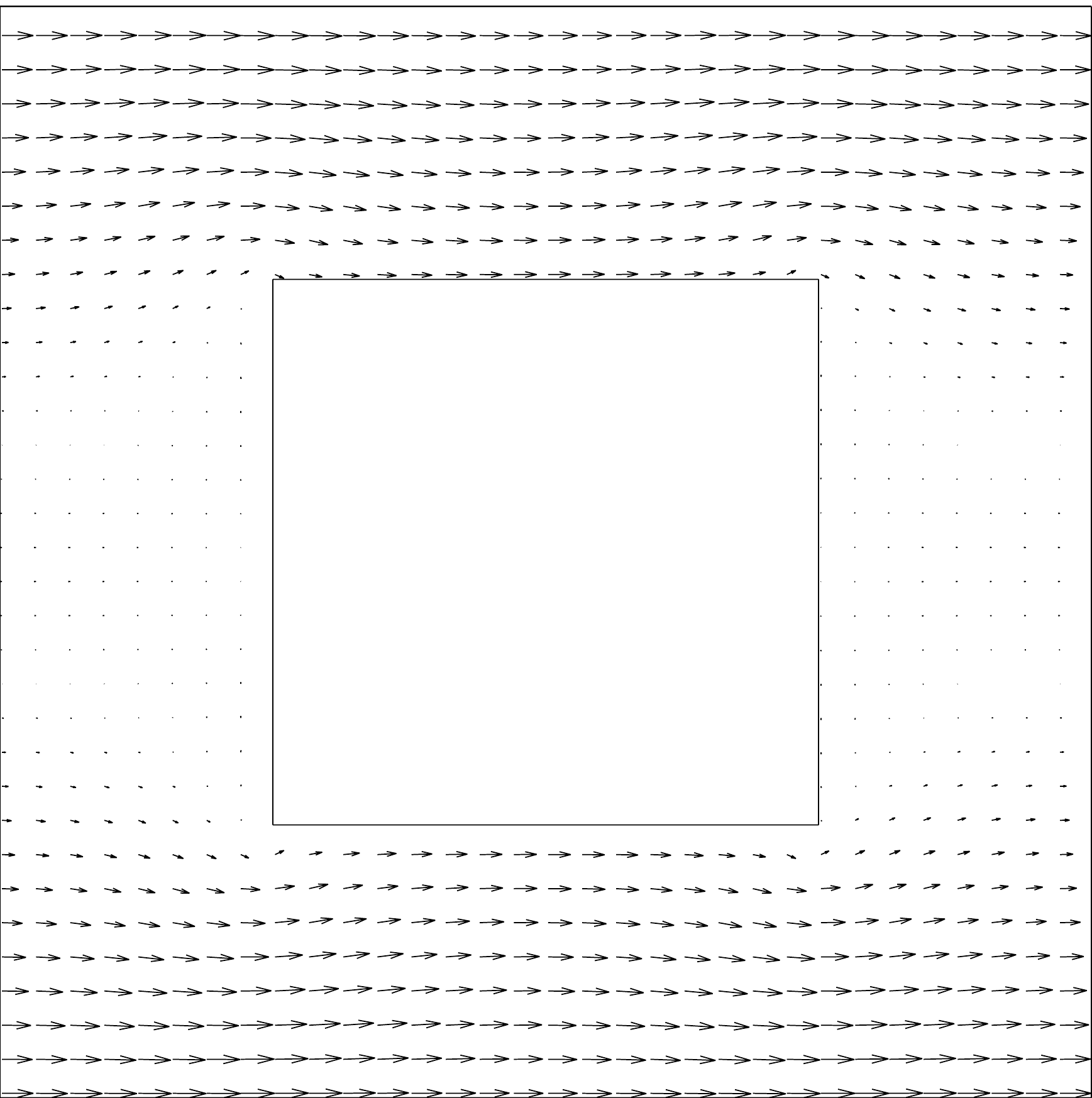}} 
    \end{picture}
\caption{ The velocity field $\b u^{1}$ for a solid obstacle $Q_2$ \eqref{rough}, and
  for a slippery bubble $Q_2$  \eqref{slippery}.  } \label{2vec}
\end{center}
\end{figure}

In Figures \ref{mu2100} and \ref{mu2vec}, we plot the two velocity fields of two-fluid flow 
  \eqref{porous} for two viscosity coefficients $\mu_2$.
When $\mu_2$ is big, the sticky inner fluid flows less and drags the outer fluid 
    near the interface.
When $\mu_2$ approaches infinity,  the inner fluid stops and it posts a zero 
  Dirichlet boundary condition for on tangential velocity of the outer fluid at
   the inner boundary $\tilde \Gamma=\partial Q_2$.
The model of a solid obstacle \eqref{rough} is a limit case of the model of two-fluid
   \eqref{porous} when $\mu_2\to\infty$.
We can compare the left chart of Figure \ref{2vec} 
   and the left chart of Figure \ref{mu2100}.

\begin{figure}[htb]\begin{center}\setlength\unitlength{1in}
    \begin{picture}(4.5,2.3)
 \put(0.0,0){\includegraphics[width=2.3in]{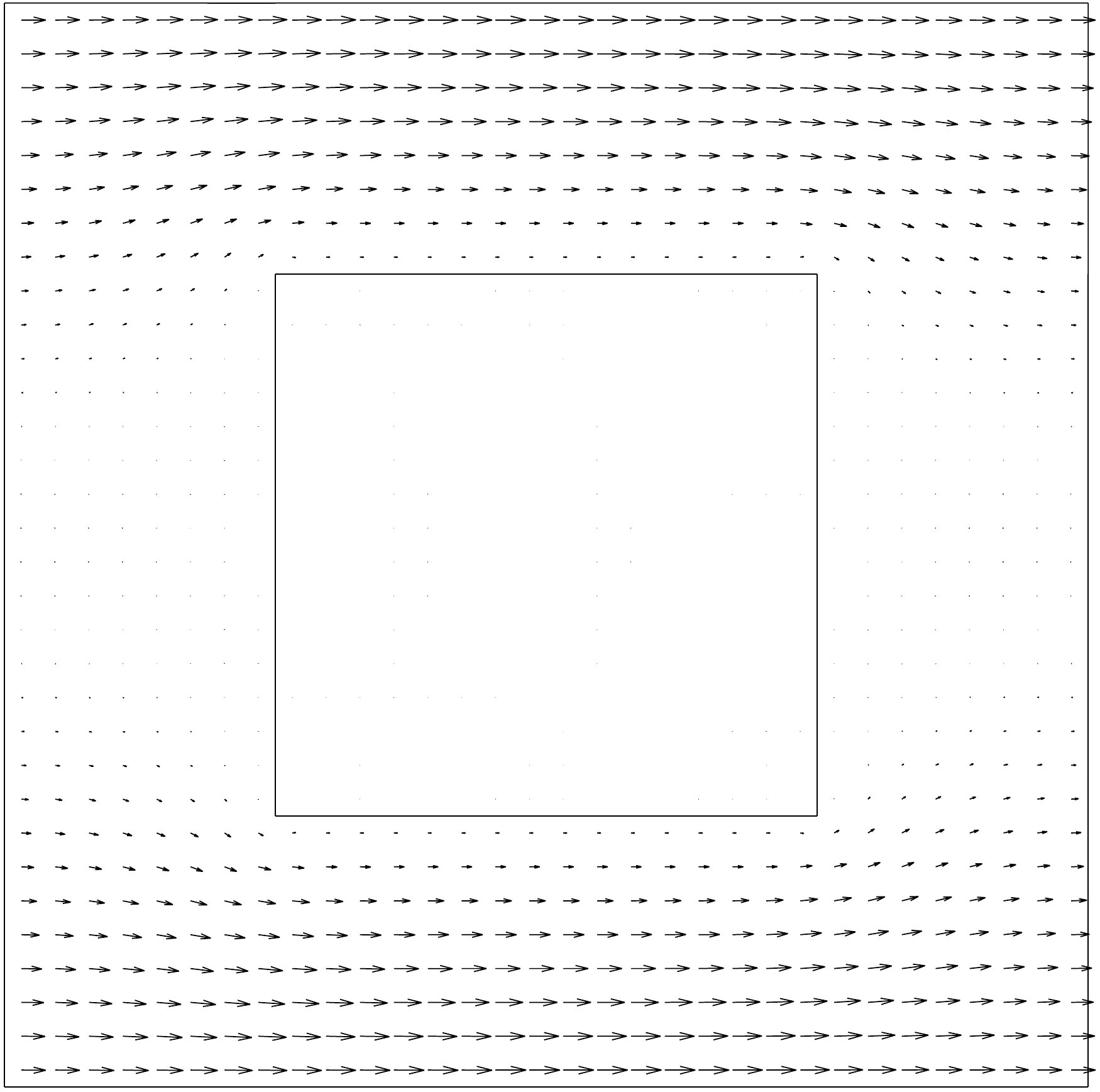}}
 \put(2.5,0.2){\includegraphics[width=1.9in]{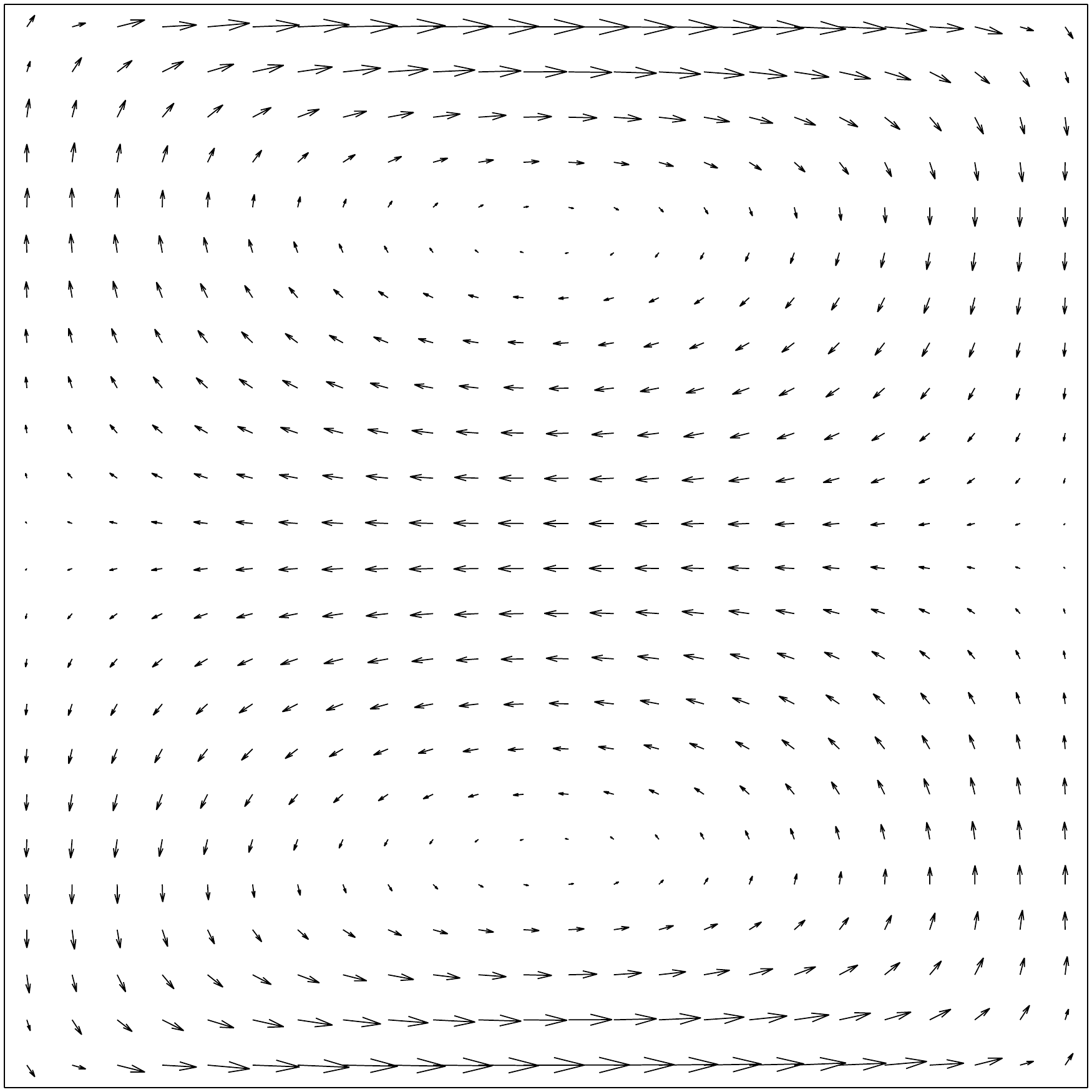}} 
    \end{picture}
\caption{ The velocity field $\b u_{3}$ for two-fluid flow
    \eqref{porous} with $\mu_2=10^{2}$ on $Q$ (left), on $Q_2$ (right, scaled by 200).  } 
   \label{mu2100}
\end{center}
\end{figure}

When $\mu_2$ approaches zero,  the inner fluid flow freely which
   produces little drag on the outer fluid.
In theory, the force inside fluid $Q_2$ may even push outer fluid somewhat.
But due to the zero outflow boundary condition on the velocity at $\partial Q_2$,
  such a force would be balanced by its left portion and right portion of
   an edge of $\partial Q_2$. 
It is equivalent to zero tangential stress boundary on the outer flow.
That is, model of a slippery bubble \eqref{slippery} is a limit 
    model of two-fluid  \eqref{porous} with $\mu_2\to 0$.
We may compare the right chart of Figure \ref{2vec} 
   and the left chart of Figure \ref{mu2vec}.

\begin{figure}[htb]\begin{center}\setlength\unitlength{1in}
    \begin{picture}(4.5,2.3)
 \put(0.0,0){\includegraphics[width=2.3in]{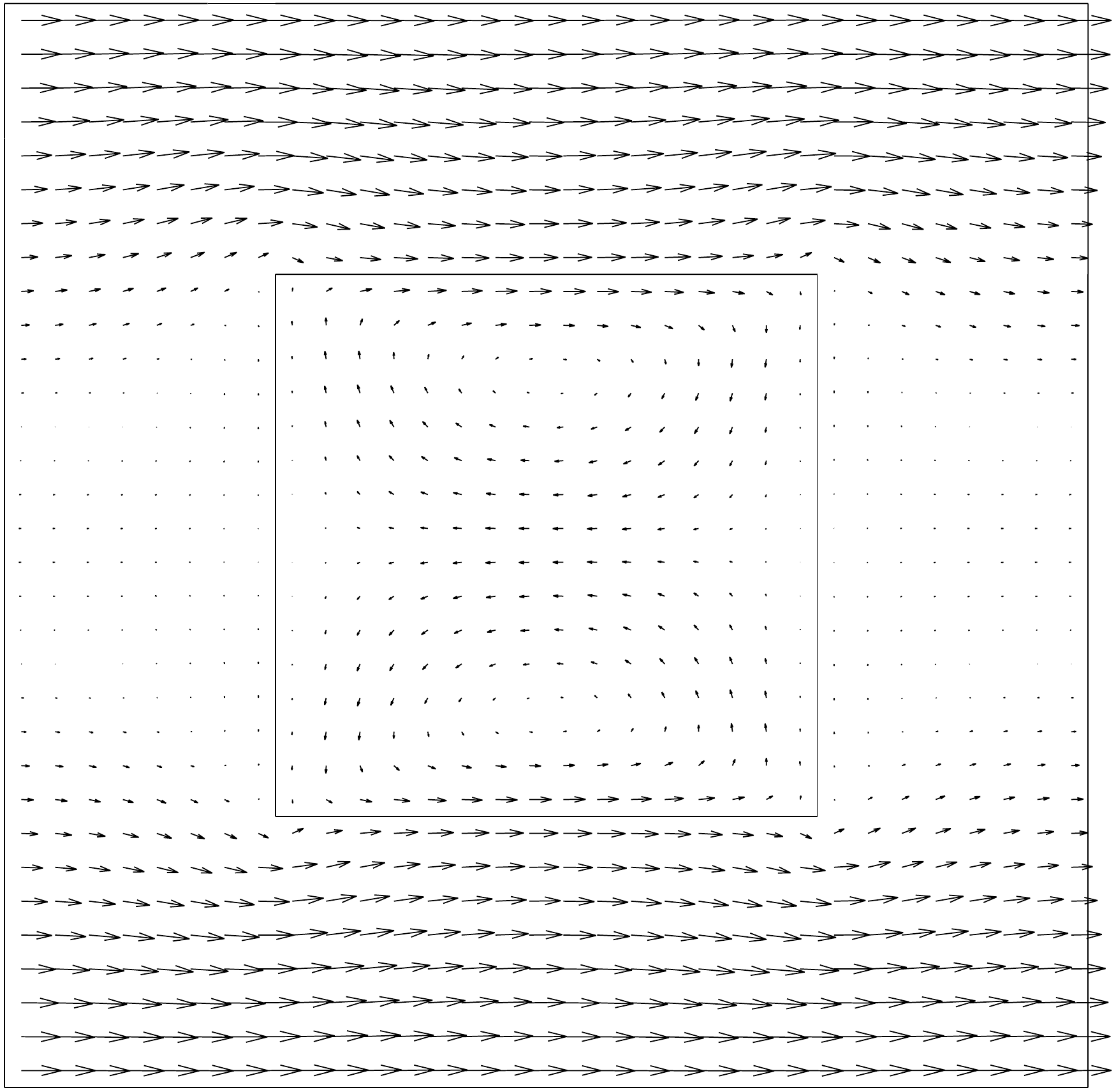}}
 \put(2.5,0.2){\includegraphics[width=1.9in]{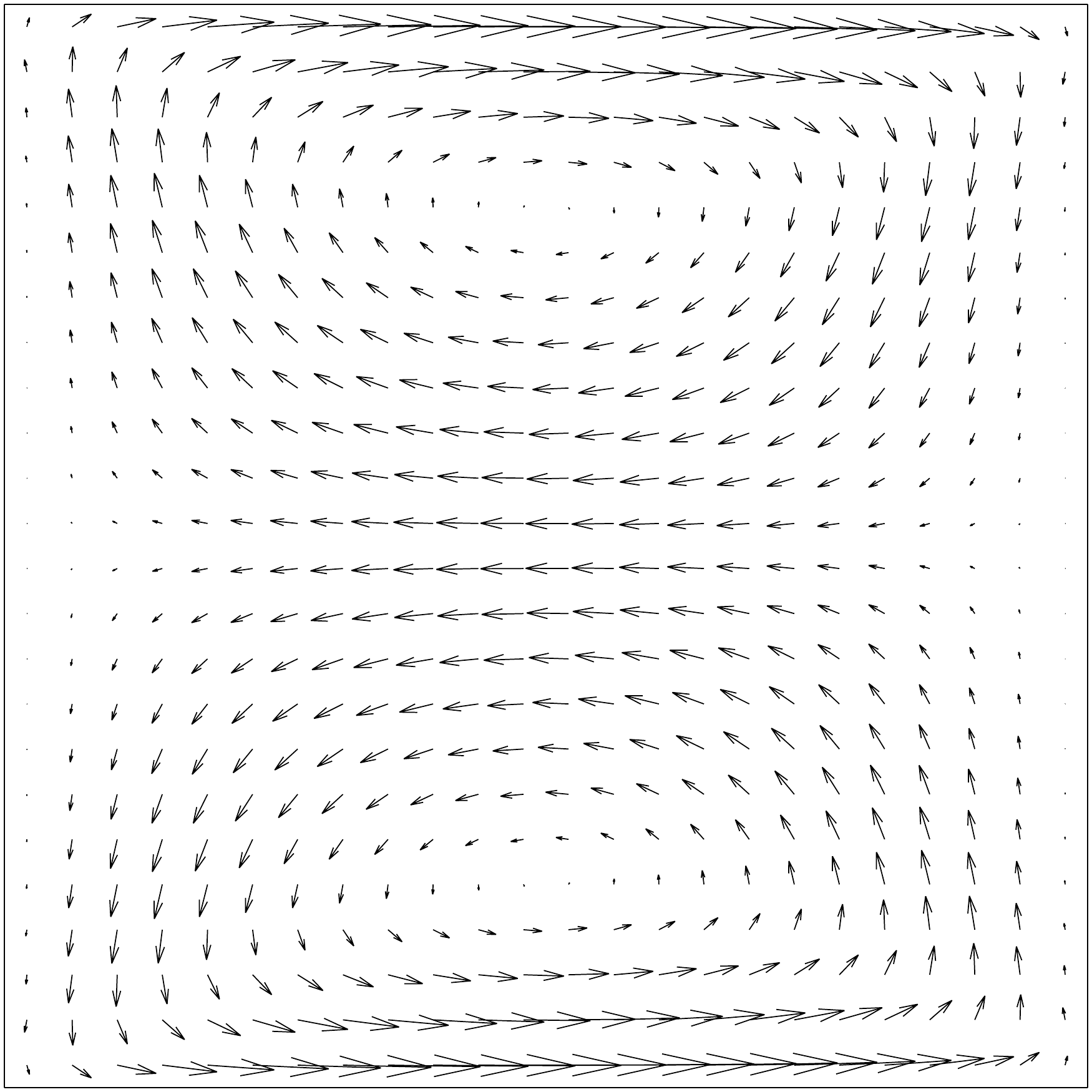}} 
    \end{picture}
\caption{ The velocity field $\b u_{3}$ for two-fluid flow
    \eqref{porous} with $\mu_2=10^{-2}$ on $Q$ (left), on $Q_2$ (right, scaled by 2).  } 
   \label{mu2vec}
\end{center}
\end{figure}

The homogenized permeability tensor $\bK=\begin{pmatrix} k_{11} & k_{12}\\ k_{21}& k_{22}
   \end{pmatrix} $ is computed by
 \an{ \label{k111} k_{11} &= \frac 1{|Q|} \int_{Q\setminus Q_2} \b u_{1} \cdot \b e_{1} d\b y
   \quad \t{ for } \eqref{rough}, \\
 \label{k112}  k_{11} &= \frac 1{|Q|} \int_{Q\setminus Q_2} \b u_{2} \cdot \b e_{1} d\b y
   \quad \t{ for }  \eqref{slippery}, \\
      k_{11} &= \frac 1{|Q|} \int_{Q } \b u_{3} \cdot \b e_{1} d\b y
   \quad \t{ for } \eqref{porous}. 
  \label{k113} }
Due to the symmetry,  in all our examples we have $ k_{11}=k_{22}$ and $ k_{12}=k_{21}=0. $
  
\begin{table}[h!]
  \centering \renewcommand{\arraystretch}{1.1}
  \caption{Computed permeability $k_{11}$ by \eqref{k111}-\eqref{k113}. }\label{t1}
\begin{tabular}{c|c|c|c|c|c}
\hline
level  &\eqref{rough} &\multicolumn{3}{c|}
 {\eqref{porous}} & \eqref{slippery}  \\ \hline  
 &  & $\mu_2=10^4$ & $\mu_2=1$  & $\mu_2=10^{-4}$ &  \\ 
\hline  
1&0.0105& 0.0105& 0.0122 &0.0140 &0.0140\\
2&0.0119& 0.0119& 0.0144 &0.0181 &0.0181\\
3&0.0125& 0.0125& 0.0154 &0.0209 &0.0209\\
4&0.0128& 0.0128& 0.0159 &0.0228 &0.0228\\
5&0.0129& 0.0129& 0.0161 &0.0240 &0.0240\\
 \hline 
\end{tabular}%
\end{table}%

To verify the convergence results stated in \eqref{large_z_asymp} and \eqref{small_z_asymp}, we solve the two-fluid problem \eqref{porous} with $\mu_1=1$ and $\mu_2=10^{-4},\,1,\, 10^4$.
 In Table \ref{t1}, this model is between the two `limiting' models \eqref{rough} and \eqref{slippery}. 
   
To see how viscosity $\mu_2$ influences the flow,  we plot $(\b u_{3})_1$ in  Figure \ref{2u1} for two different $\mu_2$ with $\mu_1 = 1$.

  \begin{figure}[h!]\begin{center}\setlength\unitlength{1in}
    \begin{picture}(4.8,1.6)
 \put(0,0){\includegraphics[width=2.4in]{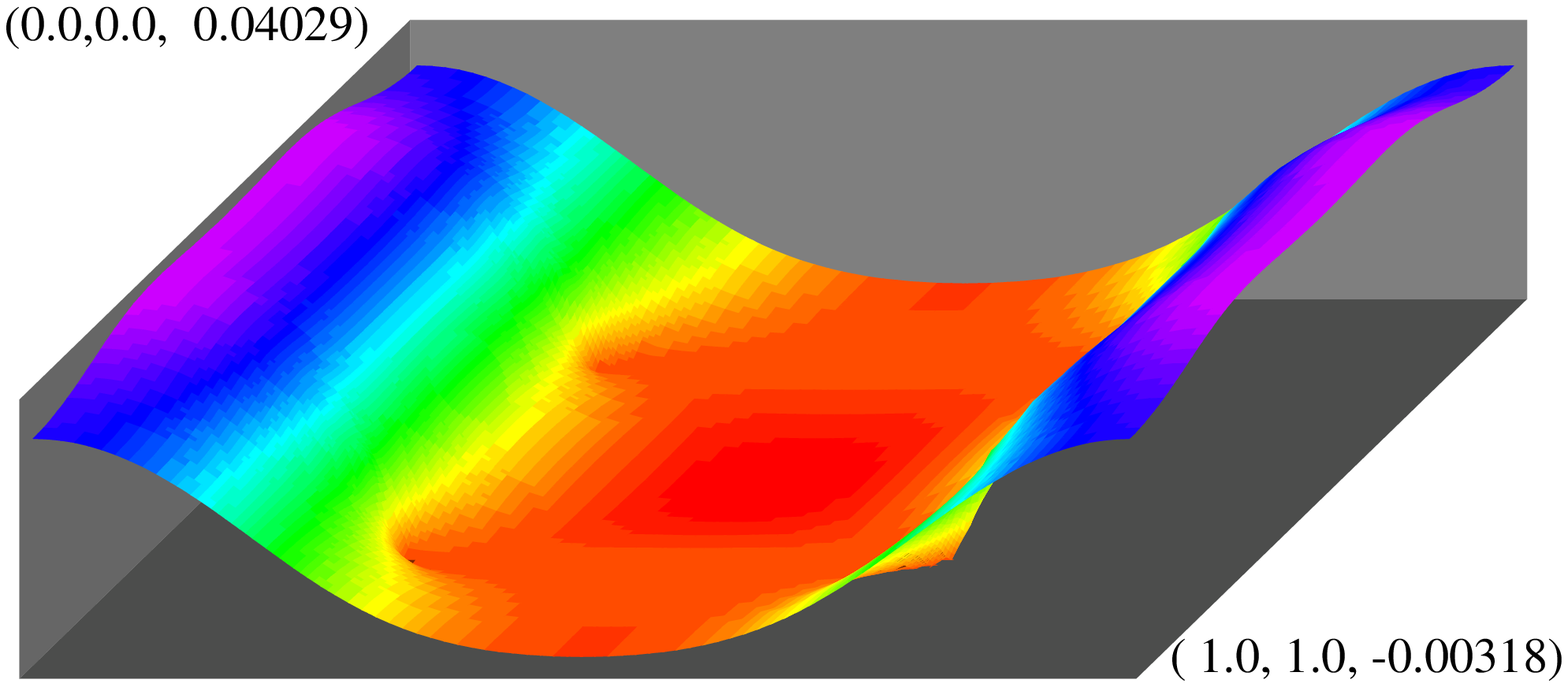}}
 \put(2.4,0){\includegraphics[width=2.4in]{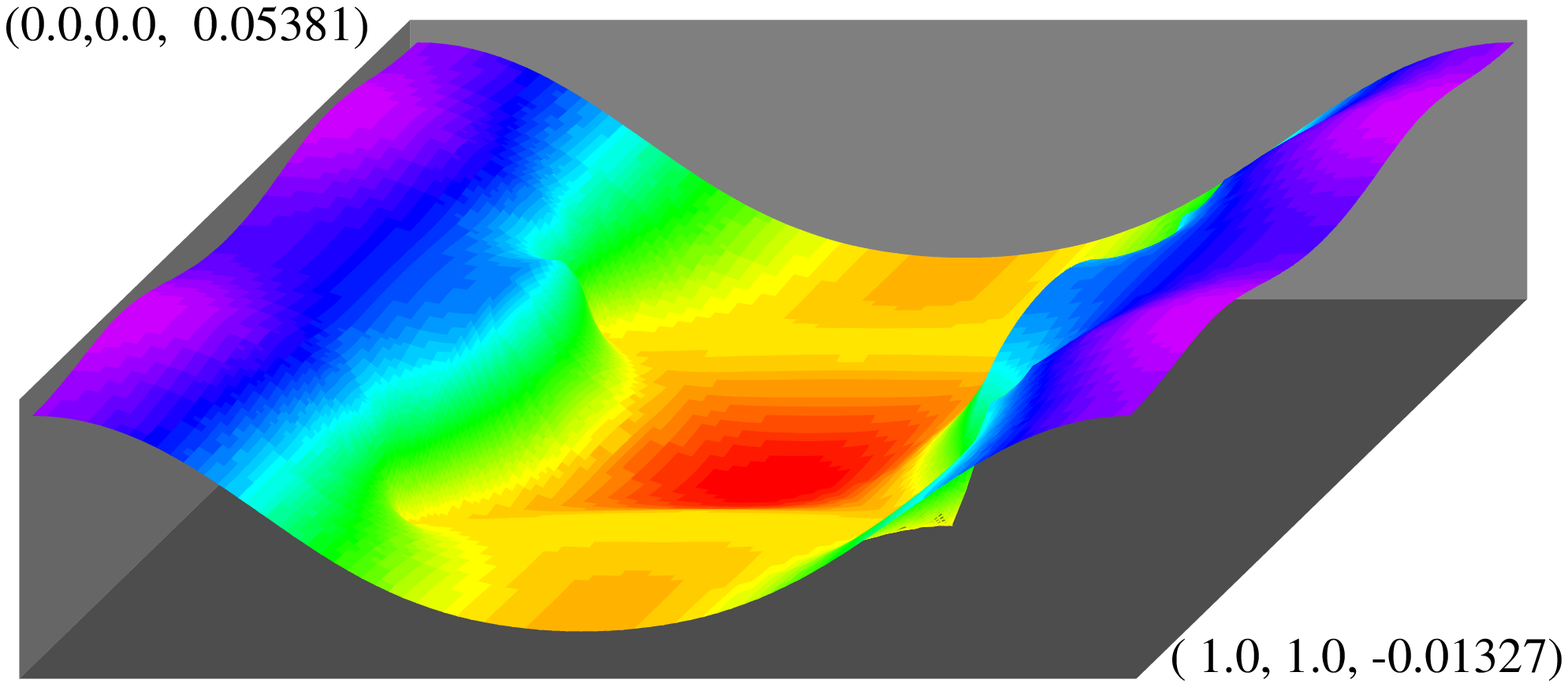}} 
    \end{picture}
\caption{ The first component of velocity $\b u_{3}$, from \eqref{porous},
    for $\mu=10^2$ and
    $\mu_2=10^{-2}$.  } \label{2u1}
\end{center}
\end{figure}

Though the magnitude of $(\b u_{3})_1$ is way larger than that of  $(\b u_{3})_2$, their corresponding stress are about the same size. In Figure \ref{2st},  we plot them for a comparison.  We plot the stress intensity $|\nabla \b u^{1}|$ in Figure \ref{2int}.
  
\begin{figure}[h!]\begin{center}\setlength\unitlength{1in}
    \begin{picture}(4.8,3.3)
 \put(0.0,0){\includegraphics[width=2.3in]{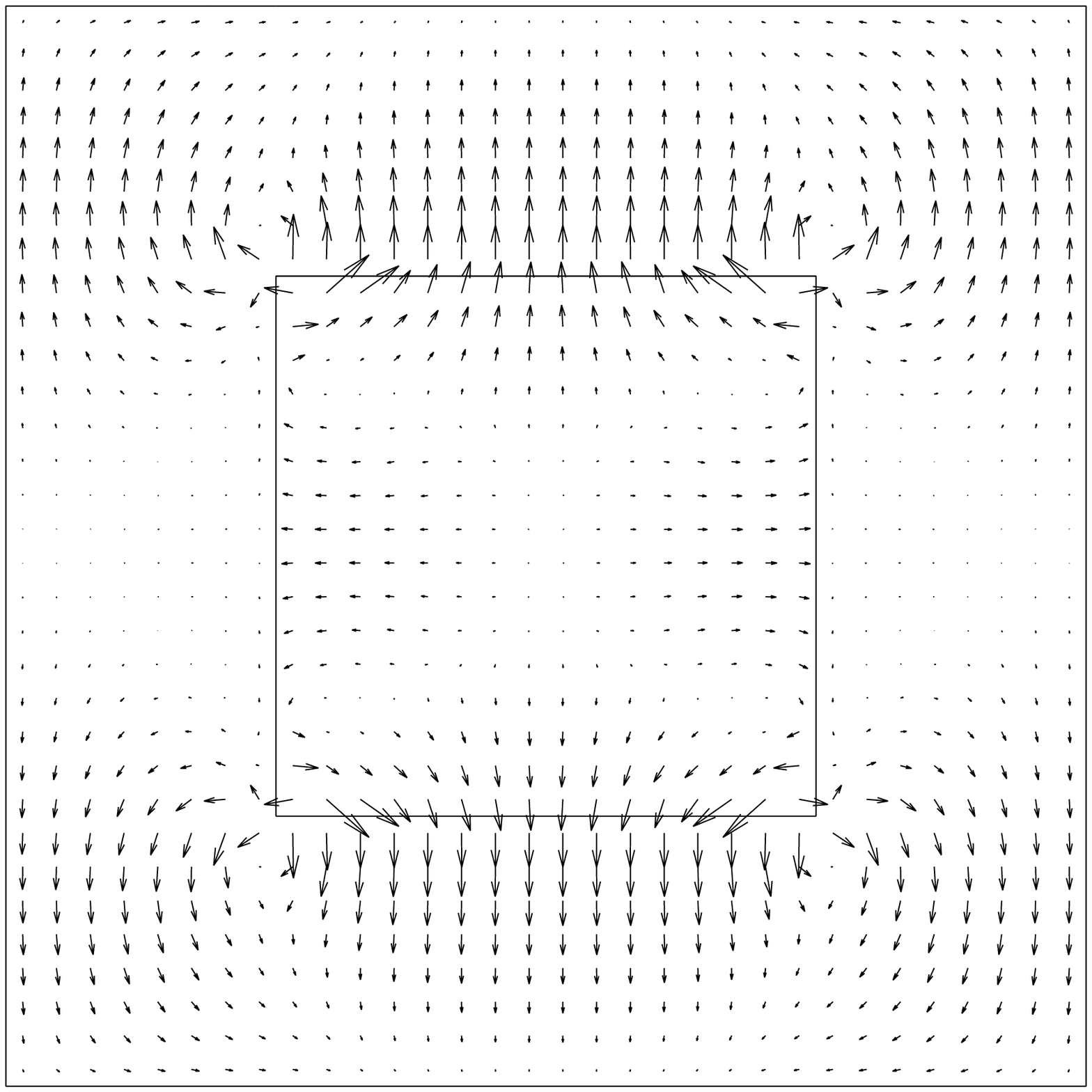}}
 \put(2.5,0){\includegraphics[width=2.3in]{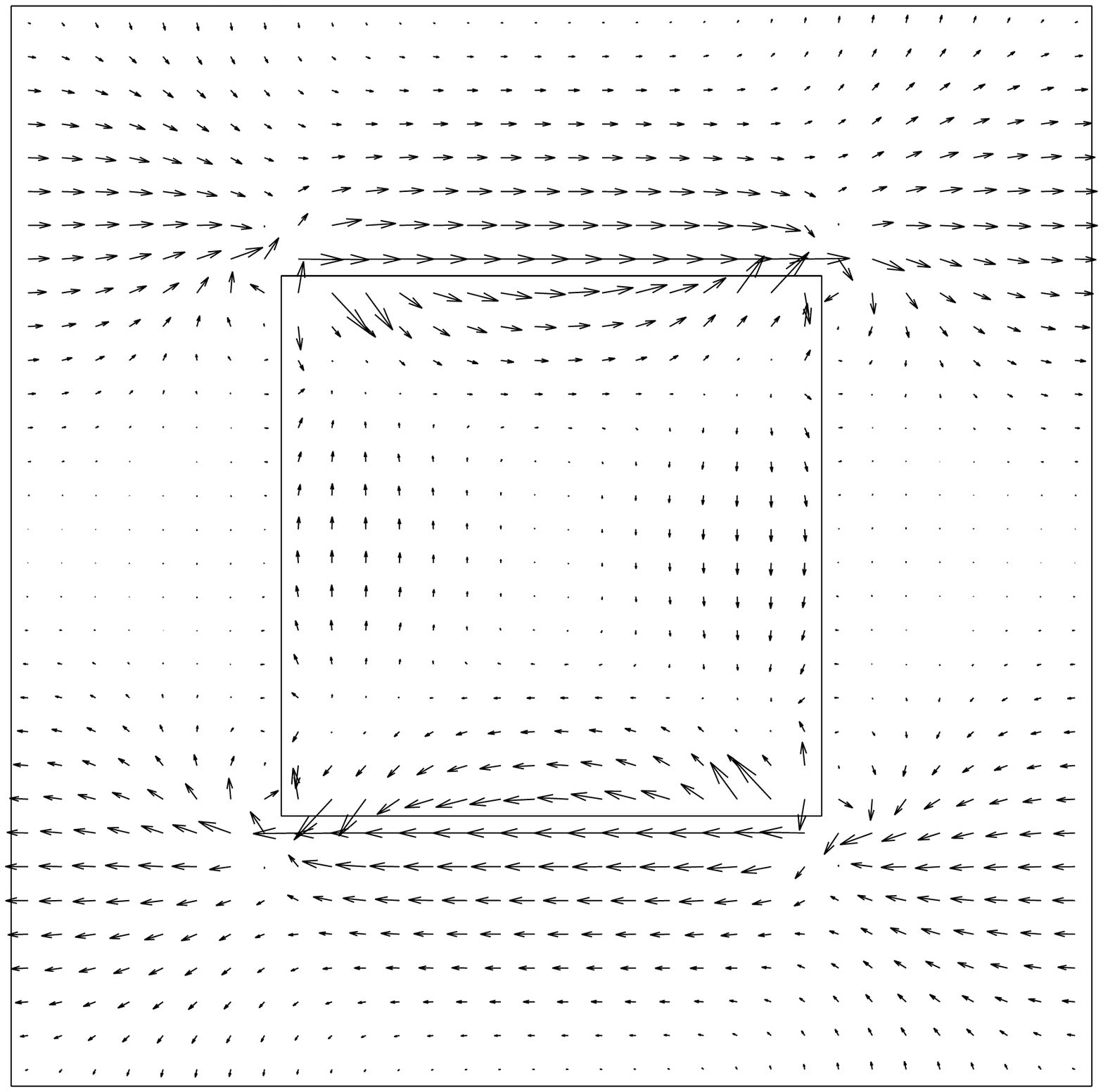}} 
    \end{picture}
\caption{ The stress  $\nabla (u_{3})_1$, and
   $\nabla (u_{3})_2$ for 
     \eqref{porous} with $\mu_2=10^{2}$.  } 
   \label{2st}
\end{center}
\end{figure}
    
\begin{figure}[h!]\begin{center}\setlength\unitlength{1in}
    \begin{picture}(5.1,1.5)
 \put(0,0){\includegraphics[width=1.5in]{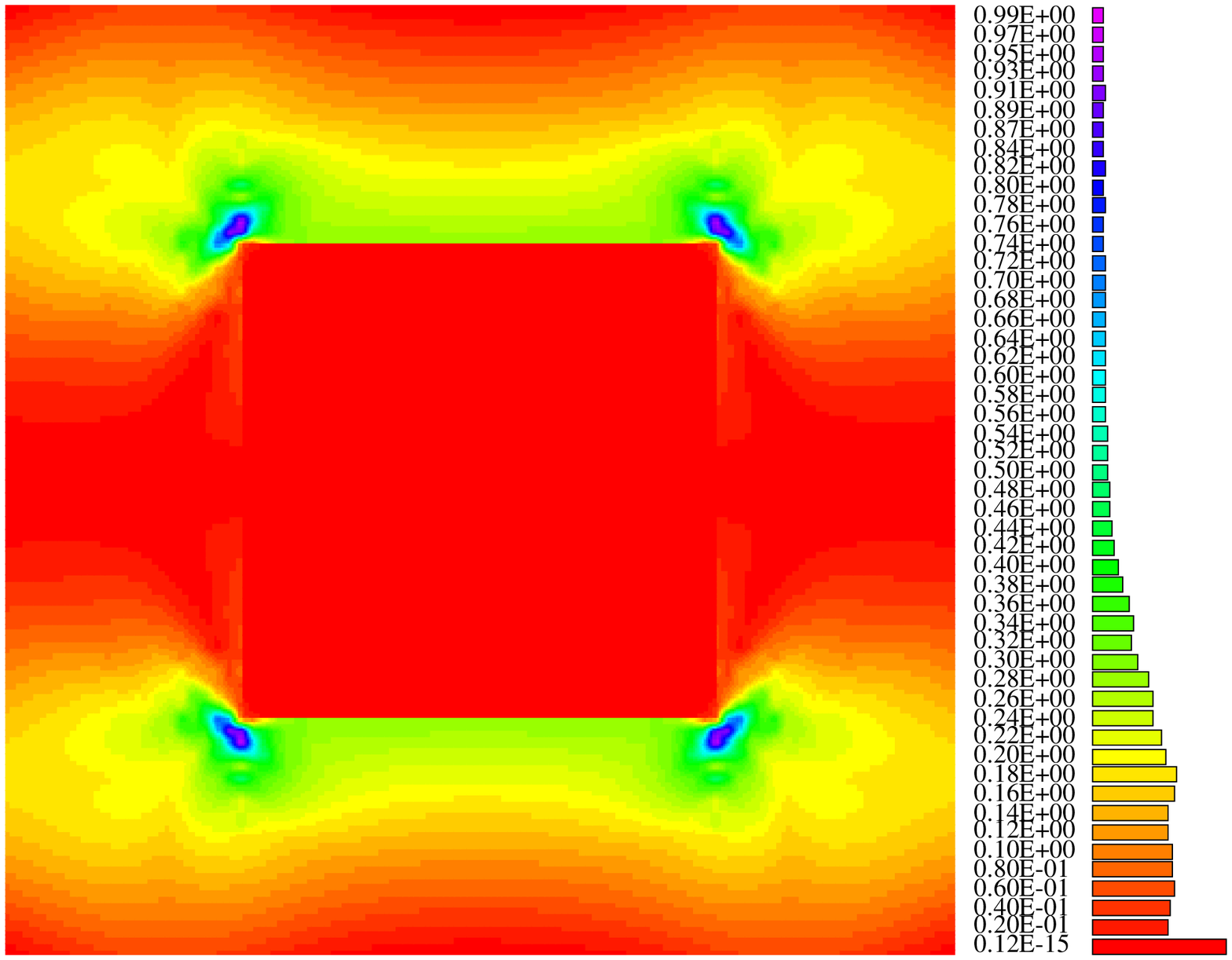}} 
  \put(1.6,0){\includegraphics[width=1.5in]{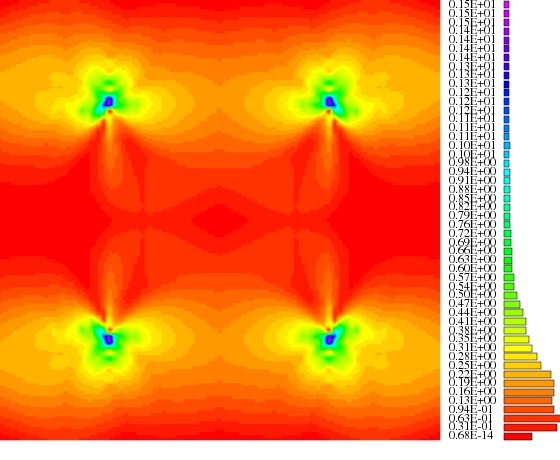}} 
  \put(3.2,0){\includegraphics[width=1.5in]{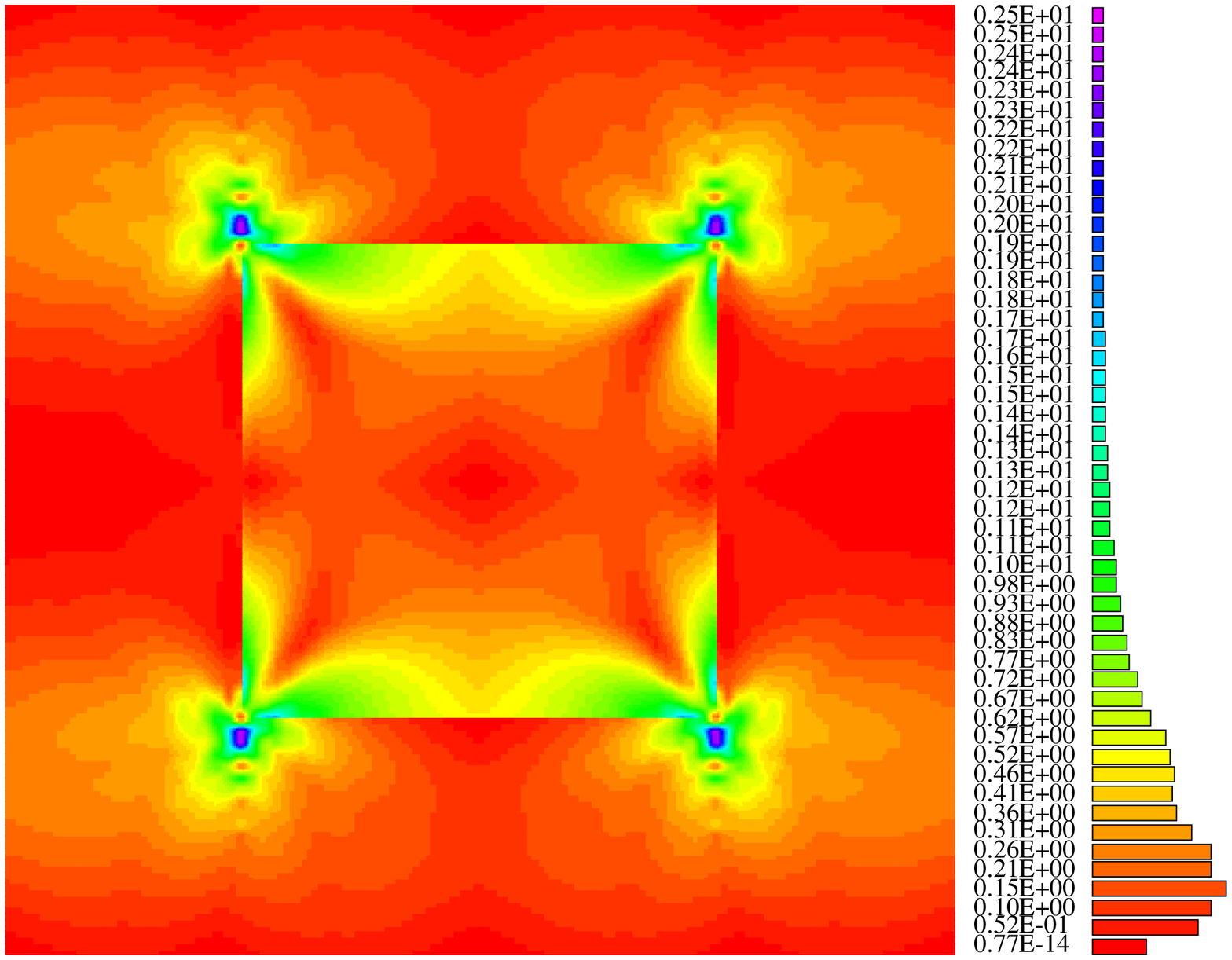}} 
    \end{picture}
\caption{ The stress intensity $|e( (\b u_3)_1 )|$
    in \eqref{porous} with $\mu_2=10^{2}$, $\mu_2=1$, $\mu_2=10^{-2}$.  } 
   \label{2int}
\end{center}
\end{figure}

Finally we compute the energy of the two-fluid flow,
\an{\label{en1} E(Q_2) &= \int_{Q_2} \mu(\b y)  e(\b u_3) :
           e(\b u_3) d \b y,  
    \\ \label{en2} E(Q) &= \int_{Q } \mu(\b y)  e(\b u_3) :
           e(\b u_3) d \b y .  }
 The homogenized permeability can also be computed by the energy,
 \an{\label{k114} k_{ij} &=\frac 1 {|Q|} \int_Q \mu(\b y)  e(\b u_3) :
           e(\b u_3) d \b y.  }
In Table \ref{t5}, we demonstrate the equivalence of these two definitions for $k_{11}$.

\begin{table}[h!]
  \centering \renewcommand{\arraystretch}{1.1}
  \caption{Computed permeability $k_{11}$ both ways and energy. }\label{t5}
\begin{tabular}{c|c|c|c|c}
\hline
level &$k_{11}$ \eqref{k113}   &$k_{11}$ \eqref{k114} & $E(Q_2)$ \eqref{en1}
          & $E(Q_2)/E(Q)$ \eqref{en2} \\ \hline  
          &\multicolumn{4}{c}{For $\b u_3$ in \eqref{porous} with 
           $\mu_2=10^2$}  \\ \hline  
 1&  0.107E-01&  0.107E-01&  0.952E-04&  0.888E-02\\
 2&  0.121E-01&  0.121E-01&  0.811E-04&  0.670E-02\\
 3&  0.127E-01&  0.126E-01&  0.676E-04&  0.534E-02\\
 4&  0.129E-01&  0.129E-01&  0.604E-04&  0.468E-02\\
 5&  0.130E-01&  0.130E-01&  0.570E-04&  0.438E-02\\  \hline  
   &\multicolumn{4}{c}{For $\b u_3$ in \eqref{porous} with 
           $\mu_2=1$}  \\ \hline  
 1&  0.122E-01&  0.122E-01& 0.752E-03&  0.615E-01\\
 2&  0.144E-01&  0.144E-01& 0.135E-02&  0.936E-01\\
 3&  0.154E-01&  0.154E-01& 0.175E-02&  0.113E+00\\
 4&  0.159E-01&  0.159E-01& 0.200E-02&  0.125E+00\\
 5&  0.162E-01&  0.162E-01& 0.215E-02&  0.132E+00\\  \hline  
  &\multicolumn{4}{c}{For $\b u_3$ in \eqref{porous} with 
           $\mu_2=10^{-4}$}  \\ \hline  
 1&  0.140E-01&  0.140E-01&  0.432E-06&  0.308E-04\\
 2&  0.181E-01&  0.181E-01&  0.109E-05&  0.602E-04\\
 3&  0.209E-01&  0.209E-01&  0.183E-05&  0.875E-04\\
 4&  0.228E-01&  0.228E-01&  0.254E-05&  0.111E-03\\
 5&  0.240E-01&  0.240E-01&  0.316E-05&  0.131E-03\\
 \hline 
\end{tabular}%
\end{table}%

\section{Conclusion and future work}
In this paper, we show that the permeability of a porous {material} \cite{tartar1980incompressible} and that of a bubbly fluid \cite{Lipton_Avellaneda_1990} are limiting cases of the complexified version of the two-fluid models posed in  \cite{Lipton_Avellaneda_1990}. We assume the viscosity of the inclusion fluid is $z\mu_1$ and the viscosity of the hosting fluid is $\mu_1$, $z\in\mathbb{C}$. The proof  is carried out by construction of solutions for large $|z|$ and small $|z|$ by an iteration process similar with the one used in  \cite{bruno1993stiffness,golden1983bounds} and analytic continuation.  Moreover, we also show that for a fixed microstructure, the permeabilities of these three cases share the same integral representation formula (IRF) \eqref{IRF_Ks_prime} with different values of $s$, as long as the 'contrast parameter' $s:=\frac{1}{z-1}$ is not in the interval  $[-\frac{2E_2^2}{1+2E_2^2},-\frac{1}{1+2E_1^2}]$, where the constants $E_1$ and $E_2$ are the extension constants that depend on the geometry of $Q_1$, $Q_2$ and $Q$. For the mixture with bubbles, $s=-1$ and thus
\begin{equation}
K^{(B)}=\bK^{(D)}+\int_{\frac{1}{1+2E_1^2}}^{\frac{2E_2^2}{1+2E_2^2}} \frac{1}{1-t}d\pmb{\sigma}(t)
\end{equation} 
Also, we note that the matrix-valued measure in \eqref{integral_rep_K_new} has a Dirac measure sitting at $\lambda=0$ with strength equal to $\bK^{(D)}$. The permeability $\bK^{(D)}$ is related to the measure in the sense that the zero-th moment of the measure is equal to $\bK(z=1)-\bK^{(D)}$.

Clearly, the positive matrix-valued measure $d\pmb{\sigma}$ is independent of $s$ and it characterizes how the geometry influences the permeability. We have shown that this measure is related to the projection measure of the self-adjoint operator $\Gamma_\chi$ and its moments can be computed by equation \eqref{moment_relation}. 

Because the IRF is valid for most of $s$ on the complex plane, the IRF will be useful in the study of two-fluid mixture with complex viscosities such as dehomogenization for these fluid. Also, the integration limits in the IRF should imply bounds on the permeability tensors. We will explore the results of this paper in these direction in the future.  

\pmb{Acknowledgement} 
The work of CB and MYO was partially sponsored by the US National Science foundation via grants NSF-DMS-1413039 and NSF-DMS-1821857.
\bibliographystyle{acm}
\bibliography{permeability_reference}

\end{document}